\documentclass[12pt]{article}

\usepackage{color,graphicx}
\usepackage{amsmath,amssymb,amsthm}
\usepackage{epsfig}\usepackage{graphicx}
\newtheorem{theorem}{Theorem}
\newtheorem{lemma}{Lemma}
\newtheorem{obs}{Observation}
\newtheorem{prop}{Proposition}

\theoremstyle{remark}
\newtheorem{ex}{Example}
\newtheorem{remark}{Remark}
\newtheorem{definition}{Definition}

\newcommand{\ie}[1]{#1}

\newcommand{\N}{\mathbb N}
\newcommand{\R}{\mathbb R}
\newcommand{\pair}[1]{\langle #1\rangle}

\date{}

\title{Goodman-Strauss theorem revisited}
\author{Nikolay Vereshchagin\\
  Moscow State University, HSE University,
  Yandex\thanks{This paper was prepared within the framework of the HSE University Basic Research Program.}}

\begin{document}
\maketitle

\begin{abstract}
The Goodman-Strauss theorem states that for ``almost every" substitution, 
the family of substitution tilings is sofic, that is, it can be defined by local rules 
for some decoration of tiles. The 
conditions on the substitution that guarantee the soficity 
are quite complicated. 

In this paper we propose a version of Goodman-Strauss theorem with simpler
conditions on the substitution.  
Although the conditions are quite restrictive,
we show that, in combination with two simple tricks (taking a sufficiently large power of the substitution and combining small
tiles into larger ones), 
our version of Goodman-Strauss theorem can also prove the soficity of the family of substitution tilings for ``almost every'' substitution. 

We also prove a similar theorem for the family of \emph{hierarchical} tilings associated with the given substitution. 
A tiling is called hierarchical if it has a composition under the substitution, such that this composition also has a composition, 
and so on, infinitely many times. Every substitution tiling is hierarchical, but the converse is not always true. 
Fernique and Ollinger 
formulated some conditions on the substitution that guarantee that the family of hierarchical tilings is sofic. 
However, their technique does not prove this statement under such general conditions as in their paper. 
In the present paper, we show that under the same assumptions, as for our version of the Goodman-Strauss theorem, 
their technique works. 
\end{abstract}

\section{Introduction}

\subsection{Substitutions and Substitution Tilings}

Assume that a  finite set of polygons $\alpha_1,\dots,\alpha_M$ is given. Some of them may be congruent; to distinguish them, we  
mark them with different colors. These polygons are called \emph{prototiles}, and \emph{tiles} are any shifts of prototiles.
Prototiles cannot be rotated or flipped.
A shift of $\alpha_i$ is called a tile \emph{of the form $\alpha_i$}.
A \emph{tiling} is any set of pairwise non-overlapping  tiles, that is, tiles having no common interior points. 
 A tiling is said to be \emph{side-to-side} 
if any two sides of its distinct tiles sharing a fragment of positive length coincide. 


In addition,
a \emph{substitution} $\sigma$ is given, that is, for each
polygon $\alpha_i$ a \emph{rule} $\alpha_i\to \mathcal{M}_i$ is given,
where $\mathcal{M}_i$ is some tiling of the polygon $\theta \alpha_i$.
Here $\theta>1$ is a real number  and multiplication by $\theta$
means stretching by $\theta$ times without rotation.
The tiling $\mathcal{M}_i$
is called 
 the \emph{decomposition} of $\alpha_i$
and is denoted by $\sigma \alpha_i$. The image of a side $a$ of $\alpha_i$ 
is denoted by $\sigma a$, those images 
are called \emph{macrosides}.
The tiles from $\mathcal{M}_i$ are called \emph{children}
of the tile $\alpha_i$, and the tile itself is their \emph{parent}.
Three examples of a substitution are shown in Fig.~\ref{f12}--\ref{f14}.
\begin{figure}[t]
\begin{center}
\includegraphics[scale=1]{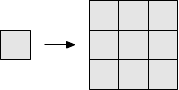}
\end{center}
\caption{First substitution}\label{f12}
\end{figure}
\begin{figure}[t]
\begin{center}
\includegraphics[scale=.6]{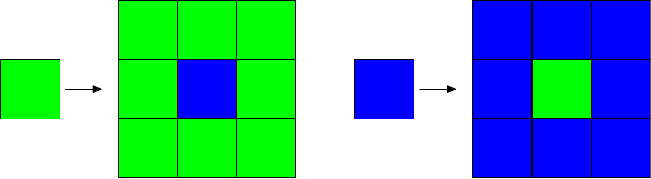}
\end{center}
\caption{Second substitution}\label{f13}
\end{figure}
\begin{figure}[t]
\begin{center}
\includegraphics[scale=.4]{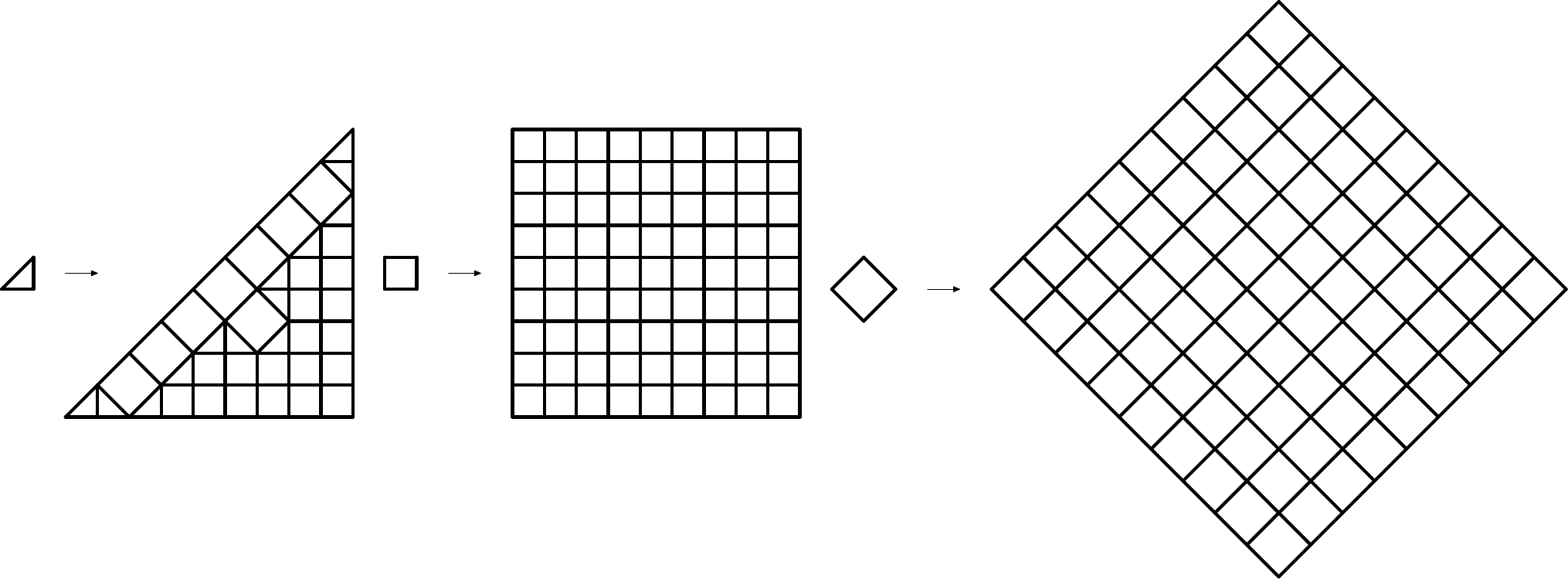}
\end{center}
\caption{Third substitution. The first prototile can be rotated by 90, 180, and 270 degrees.
Therefore, the total number of prototiles is 6.}\label{f14}
\end{figure}

The action of substitution on tilings is defined as follows: the
rules are applied to all tiles of the tiling simultaneously, i.e., the tile $\alpha_i$ is replaced by the tiling $\mathcal M_i$.
The resulting tiling is called the \emph{decomposition} of the original tiling.
The inverse operation is called the \emph{composition}, the composition of a tiling
may not exist and may not be unique.

\emph{Supertiles} are finite tilings obtained from prototiles by several 
decompositions. The number of decompositions is called the \emph{order} of the supertile.
For example, all macrotiles are supertiles of the first order.
See  Fig.~\ref{supertiles} for another example.
\begin{figure}[t]
\begin{center}
\includegraphics[scale=.8]{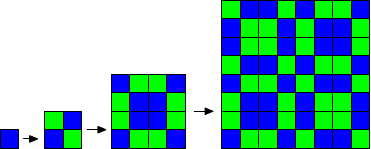}\hskip 2cm \includegraphics[scale=.8]{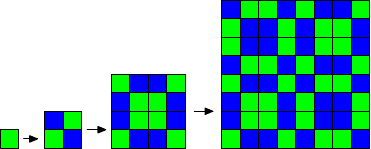}
\end{center}
\caption{The figure shows a substitution and supertiles of order 1, 2 and 3.}\label{supertiles}
\end{figure}
A tiling is called \emph{a substitution tiling} (for a given substitution)
if every its finite part is included in some supertile.

A family of tilings $\mathcal F$  is called   \emph{sofic} if 
one can color the sides of the polygons $\alpha_1,\dots,\alpha_M$, each prototile can be decorated in several ways, so that the following holds.
Call a tiling with decorated tiles \emph{proper} if it is side-to-side
and the colors of sides shared by two tiles match.
Consider the operation $\pi$ of erasing colors on tilings with decorated tiles. 
We say that a tiling with tiles from the set $\{\alpha_1,\dots,\alpha_M\}$ 
is a \emph{projection} of a tiling $\mathcal T'$ with decorated tiles
if $\mathcal T=\pi( \mathcal T')$.
Then a tiling  is in $\mathcal F$ 
if and only if it is a projection of a proper tiling with decorated tiles.

\subsection{Mozes' Theorem}
Mozes' theorem~\cite m states that under certain conditions on the substitution
the family of substitution tilings is sofic.
The assumptions on the substitution in Mozes' theorem are that all prototiles are squares of the same size 
and one more condition that we find difficult to state here.

\subsection{Goodman-Strauss Theorem}
Goodman-Strauss theorem generalizes Mozes' theorem. It
states the same thing (the soficity of
the family of substitution tilings) but under weaker conditions
on the substitution.
Here is how this theorem is formulated in~\cite{gs}:
\emph{Every substitution tiling of $\R^d$, $d > 1$, can be enforced with finite
matching rules, subject to a mild condition: the tiles are required to admit a set of ``hereditary edges''
such that the substitution tiling is ``sibling-edge-to-edge''}. In this quotation, the phrase
``Every substitution tiling of $\R^d$, $d > 1$, can be enforced with finite matching rules'' in our terminology means: ``Every family of substitution tilings is sofic''.
But the second part of the formulation, which talks about sufficient conditions for this,
is quite complicated, see Section 1.4 of~\cite{gs} on pp.  181--182.  
So we find it difficult to provide a precise formulation of the theorem.

\subsection{Fernique --- Ollinger Construction}

A tiling $\mathcal{T}$
is called \emph{hierarchical} (for a given substitution)
if there exists an infinite sequence of  side-to-side tilings
$\mathcal{T}_0=\mathcal{T},\mathcal{T}_1,\mathcal{T}_2,\dots$ in which for all $i$ the tiling $\mathcal{T}_{i+1}$ is a composition of $\mathcal{T}_i$.

Fernique --- Ollinger theorem states that, under some conditions on the given substitution, the set of hierarchical
tilings is sofic.\footnote{In fact, their statement is more general, they consider a broader class of substitutions
--- the so-called ``combinatorial substitutions".} They formulate those conditions explicitly, 
but they are not sufficient for their construction to work.
In Sections~\ref{sa2} and~\ref{sa3} we present two substitutions for which Fernique --- Ollinger construction does not work,
even though all their conditions are satisfied.

\subsection{Our contribution}
In one sentence, our contribution can be summed up as follows: we
improved Fernique --- Ollinger construction
to work under quite general  assumptions,
established how to modify the construction to prove the soficity of substitution tilings under the same assumptions,
and realized that, in combination with some simple tricks, the resulting theorem proves the soficity of families of substitution and
hierarchical tilings for almost every substitution.

In more detail, in this paper we do the following:

(a) We formulate
sufficient assumptions on the given substitution (conditions R0--R8 below) under which
the family of hierarchical tilings is sofic.
The conditions are stronger than those of Fernique --- Ollinger, but we have a complete proof,
unlike~\cite{fo}, where the proof is only sketched.
Moreover, the Fernique --- Ollinger construction does not work under the assumptions as in their paper.

(b) We show that under the same assumptions R0--R8 the family of 
 substitution tilings is sofic.

(c) Our assumptions  R0--R8 are quite restrictive. But we show
how to extend the theorem to ``almost any''
substitution using two simple tricks. For example, we prove the following generalization of Mozes’ theorem
(Theorem~\ref{th22} below):  \emph{If   all prototiles are squares of the same size
then the families of hierarchical and substitution tilings are
sofic.}


\emph{The structure of the paper.} 
To present our technique,  we first apply it to prove our generalization of Mozes' theorem (Section 2).
In Section 3, we formulate our assumptions  R0--R8.
In Section 4, we state and prove the main result. In Section 5 we  outline its scope of application.
In the Appendix we discuss the technique and explain why the construction of
Fernique --- Ollinger does not work under the assumptions as in their paper.

\section{A Generalization of Mozes' Theorem}
\begin{theorem}\label{th22}
Assume that all prototiles are squares of the same size. 
 Then (a) the family of hierarchical tilings is  sofic and (b) the family of  substitution tilings  is sofic.
\end{theorem}

The rest of this section consists of the proof of this theorem. Let $\tau$ denote given substitution 
where all prototiles are squares of the same size.
Our plan is the following: we prove first that macrotiles for some power $\tau^m$ of $\tau$ 
has  certain property and then use this property to prove the statement.

\subsection{The First Step}
Fix a natural $m$ and  
let  $\sigma$ denote $\tau^m$. 
Every $\sigma$-macrotile $\mathcal  M_i$  is a 
grid of some size $k\times k$ that does not depend on $i$ (but depends on $m$).
A \emph{markup} of $\sigma$ is a number $l$ such that $1<l<k$.
Tiles from the union of $\sigma$-macrotiles $\mathcal{M}_1,\dots,\mathcal{M}_M$ will be called \emph{$\sigma$-types}.
We call the type in $\mathcal  M_i$ that is located in $l$th row and $l$th  column
the \emph{central $\sigma$-type} in $\mathcal  M_i$. 
\begin{center}
\includegraphics[scale=1.]{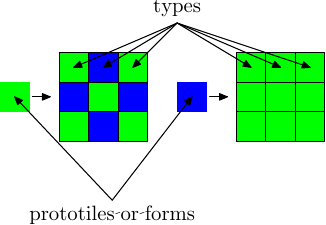}
\end{center}

\begin{definition}
For a central $\sigma$-type $t$ let $c_\sigma(t)$ denote the central type in the macrotile $\sigma t$;
the type $c_\sigma(t)$ depends only of the form of $t$.
We call a central type $t$ \emph{cyclic} if $c_\sigma(t)=t$.
\end{definition}
Here is the requirement for the markup mentioned above:
\begin{itemize}
\item \label{r9}
``Cyclicity Requirement''.   
\emph{For any central type $u$ we have $c_\sigma^{2}(u)=c_\sigma(u)$,
that is, the type $c_\sigma(u)$ is  cyclic.} See Figs.~\ref{f34} and~\ref{f1-5}. 
\end{itemize}
\begin{figure}
\begin{center}
\includegraphics[scale=1]{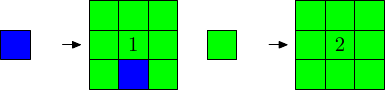}\hskip 2.5cm \includegraphics[scale= 1]{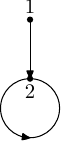}
\end{center}
\caption{An example of substitution for which there is an acyclic 
central type (type 1) but the Cyclicity Requirement is satisfied.  Central types are in the middle. 
On the right there is the graph of the function $c_\sigma(u)$.}\label{f34}
\end{figure}
\begin{figure}
  \begin{center}
    \includegraphics[scale=.6]{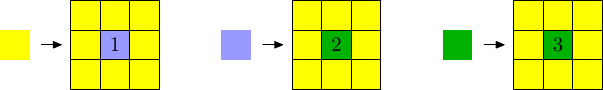}
  \end{center}
  \caption{A substitution that does not satisfy Cyclicity Requirement:
the central type 2 is acyclic and $c_\sigma(1)=2$.
}\label{f1-5}
\end{figure}

We start with the following
\begin{lemma}\label{l6}
For some $m$ 
there is a markup of $\sigma=\tau^m$ meeting
the Cyclicity Requirement.
\end{lemma}
\begin{proof}
We first demonstrate this for the substitution $\tau$ in Fig.~\ref{f1-5}.
Consider the square of this substitution (Fig.~\ref{f1-7}).
\begin{figure}
  \begin{center}
    \includegraphics[scale=1.1]{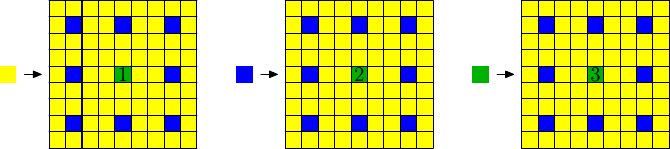}
  \end{center}
  \caption{The square of the substitution from Fig.~\ref{f1-5}.}\label{f1-7}
\end{figure}
Now all central types are mapped to the cyclic type 3.

In the general case, we 
first choose a power $\sigma$ of $\tau$ such that
each $\sigma$-macrotile  has at least 3 rows and 3 columns that is $k\ge 3$.
W.l.o.g. assume that $\tau$ itself has this property.
First try to set $\sigma=\tau$ and $c=2$. 
If the function $c_\tau$ satisfies Cyclicity Requirement, we are done.

Otherwise we claim that  
for some $m>0$ we have $c_\tau^{2m}(t)=c_\tau^m(t) $ for all central $\tau$-types $t$. In other words, $c_\tau^m(t)$ is a fixed point of the function $c_\tau^m$ for all $t$.
Indeed, the number of functions mapping central types to central types is finite. Therefore 
for some $n$ and $m>0$ we have $c_\tau^n=c_\tau^{n+m}$.
It follows that $c_\tau^{n+i}=c_\tau^{n+m+i}$ for all $i\ge 0$.
If $m\ge n$, then we set $i=m-n$ and obtain $c_\tau^{m}=c_\tau^{2m}$.
It remains to note that $m$ can be made arbitrarily large,
since the equality $c_\tau^n=c_\tau^{n+m}$ implies the equalities $c_\tau^n=c_\tau^{n+m}=c_\tau^{n+2m}=c_\tau^{n+3m}=\dots$.

Choose any $m>0$ with  $c_\tau^{2m}=c_\tau^m$ and let $\sigma=\tau^m$.
Then define the markup $l_j$ of $\tau^j$-macrotiles for $j=1,\dots,m$
recursively:\\ 
(1) $l_1=2$.\\
(2) The $i$th $\tau^{j+1}$-macrotile
  is obtained from $i$th $\tau^{j}$-macrotile 
 by replacing each tile by the corresponding  $\tau$-macrotile.
We choose the central row in $\tau^{j+1}$-macrotiles  
as the second row in the $\tau$-macrotile  
which replaces the central type of $\tau^{j}$-macrotiles.
In other words, $l_{j+1}=(l_j-1)k+2$.

What is the function $c_{\tau^m}$? It acts on central $\tau^m$-types
in the same way as $c_\tau^m$ acts on central  $\tau$-types. It follows that 
$c_\sigma(t)$ is a  cyclic type
for all central $\sigma$-types $t$. 
\end{proof}

By Lemma~\ref{l6}, to establish Theorem~\ref{th22},  it suffices to prove the following
\begin{theorem}\label{th0}
Assume that all prototiles for a substitution $\tau$ are squares of the same size and
some power $\sigma=\tau^m$ of $\tau$ has a markup meeting Cyclicity Requirement.
Then  (a) the family of hierarchical tilings for $\tau$ is  sofic and (b) the family of  substitution tilings for $\tau$ is sofic.
\end{theorem}

The rest of the section is the proof of this theorem.

\subsection{Proof  of Theorem~\ref{th0}(a)}

We will distinguish 
three kinds of sides of $\sigma$-types: 
\begin{enumerate}
\item  \emph{Inner sides}, these are sides 
whose both ends do not lie on the boundary of the macrotile. 
\item Sides that lie on the boundary of macrotile, these are called \emph{outer sides}. 
\item The remaining sides have one or two ends on the boundary but do not lie on it,
those sides are called \emph{border sides}.
\end{enumerate}
Border sides in Fig.~\ref{f66} are traversed by the green ring. 
\begin{figure}
\begin{center}
\includegraphics[scale=1.]{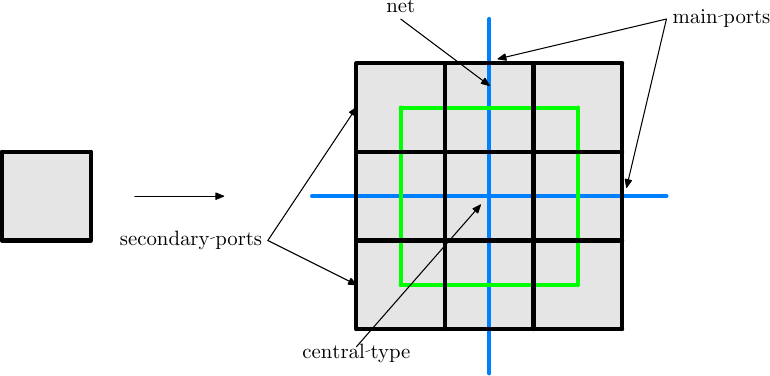}
\end{center}
\caption{Ports, central types and net paths.}\label{f66}
\end{figure}
The blue cross with the center in the central type  is called \emph{the net}.
It consists of four straight line segments $P_{\text{north}}$, $P_{\text{east}}$,
$P_{\text{south}}$, $P_{\text{west}}$
called \emph{net paths}.
The sides it traverses are called \emph{net sides}. 
All sides of the central type are thus net sides.
Outer net sides are called \emph{main ports} and the remaining 
outer sides are called \emph{secondary ports}.
Thus each net path connects a side of the central tile with a main port.
We will use terms ``north, ``east'',
``south'', ``west'' as \emph{names}
of sides of tiles. The term $z$th \emph{side of a tile} for $z\in\{\text{north, east, south, west}\}$ 
has the obvious meaning.

For each prototile $\alpha$ we define several its ``clones'' obtained by choosing colors
on the sides of $\alpha$. Then, on the tiles of the resulting set, we define a non-deterministic substitution $\sigma'$ such that:
\begin{itemize}
\item $\pi(\sigma'(A))=\sigma(\pi(A))$
for all decorated tiles $A$. Recall that $\pi$ denotes the operation of erasing colors from sides of tiles.\footnote{This is understood as follows:
for all macrotiles $\mathcal M$ that are images of $A$ under $\sigma'$ it holds 
$\pi(\mathcal M)=\sigma(\pi(A))$}
\item every $\tau$-hierarchical tiling 
  is a projection of a proper tiling with decorated tiles, and
\item every proper tiling $\mathcal T$ with decorated tiles  has a composition $\mathcal T'$ under $\sigma'$
which  is again a proper tiling; the tiling  $\mathcal T'$ has the following feature:
all tilings $\tau^{1} \pi(\mathcal T'), \dots, \tau^{m-1} \pi(\mathcal T')$
are side-to-side.
\end{itemize}
The last property guarantees that every proper tiling $\mathcal T$ with decorated tiles is $\sigma'$-hierarchical. Moreover,
for every proper tiling $\mathcal T$ its projection
 $\pi(\mathcal T)$  is $\tau$-hierarchical.
Indeed, let 
$$
\mathcal T_0=\mathcal T, \mathcal T_m, \mathcal T_{2m}, \dots
$$ 
be a sequence of proper tilings in which each tiling is a
composition of the previous one with respect to $\sigma'$.
Then in the sequence 
$$
\pi(\mathcal T_0),\tau^{m-1}\pi(\mathcal T_m),\dots,\tau^{1}\pi(\mathcal T_m),\pi(\mathcal T_m),
\tau^{m-1}\pi(\mathcal T_{2m}),\dots,\tau^{1}\pi(\mathcal T_{2m}),\pi(\mathcal T_{2m}),\dots
$$
each tilings is a $\tau$-composition of the preceding one. And due to the feature of $\mathcal T'$ all tilings in
this sequence are side-to-side.

\subsubsection{Decoration of prototiles: the general plan}
On each side of a
tile there will be three indices,
\emph{red, blue and green}, the triple of these indices constitutes the
color of the side. One more index will be written in the middle of the tile,
we will call it \emph{central}. 
This index does not affect the connection of tiles,
so we will remove it later. All indices will range through some finite sets,
which will be clear from the further presentation.
In  Fig.~\ref{f1} the decoration is shown for the first substitution assuming that $m=1$ and hence $\sigma=\tau$.   
\begin{figure}
\begin{center}
\includegraphics[scale=.9]{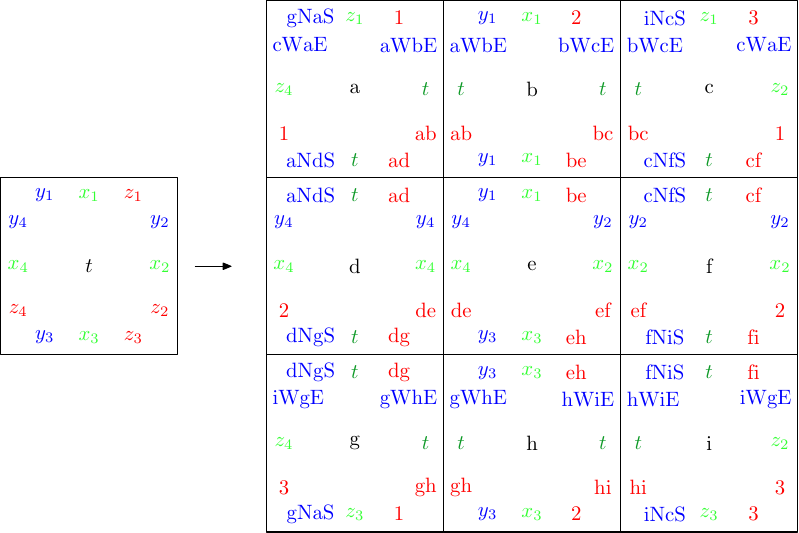}
\end{center}
\caption{
On the right there are 9 tiles assembled into a $\sigma$-macrotile where 
$\sigma=\tau$ is the first substitution. On the left there is 
the composition of that macrotile.
Types are represented by Latin letters.  2-letter strings  
represent pairs of types and 4-letter strings represent 
quadruples. Variable $t$ ranges over types
and variables $x_i,y_i,r_i$ range over  some finite sets.}\label{f1}
\end{figure}


The central index of a tile indicates in which
$\sigma$-macrotile this tile can be located, and at what place in the macrotile,
when partitioning a proper tiling into $\sigma$-macrotiles.
Therefore, we will also call this index the \emph{type}
of the decorated tile. Red indices force tiles of a proper tiling
to assemble into $\sigma$-macrotiles.
To this end, the red index on any inner or border  side of a tile $A$ consists of the pair
$\langle $the type of $A$, the type of the neighbor of $A$ on this side in the macrotile$\rangle$.
Red indices ensure also that
these macrotiles
connect macroside-to-macroside:
on an outer side,
the red index includes  the number of that side
in the order of traversing this macroside from left to right.

The tuple of red indices of a tile is called its
\emph{red contour}.
From the red contour of a tile we can compute its
 central index and the other way around.
Thus the red contour  has the same information 
as the central index. 

Why are red indices not enough?  Some tiles of each macrotile of a proper tiling
must carry information
about the type of the tile obtained from that macrotile via  composition. 
Similarly in each order-2 supertile there must be tiles carrying information about the type of the tile obtained from that supertile tile via the 2-fold composition. And so on.
This information will be called  \emph{global}, it  
is represented 
by green indices of some tiles. 
Those tiles are shown in Fig.~\ref{fgr}.
\begin{figure}
\begin{center}
\includegraphics[scale=.25]{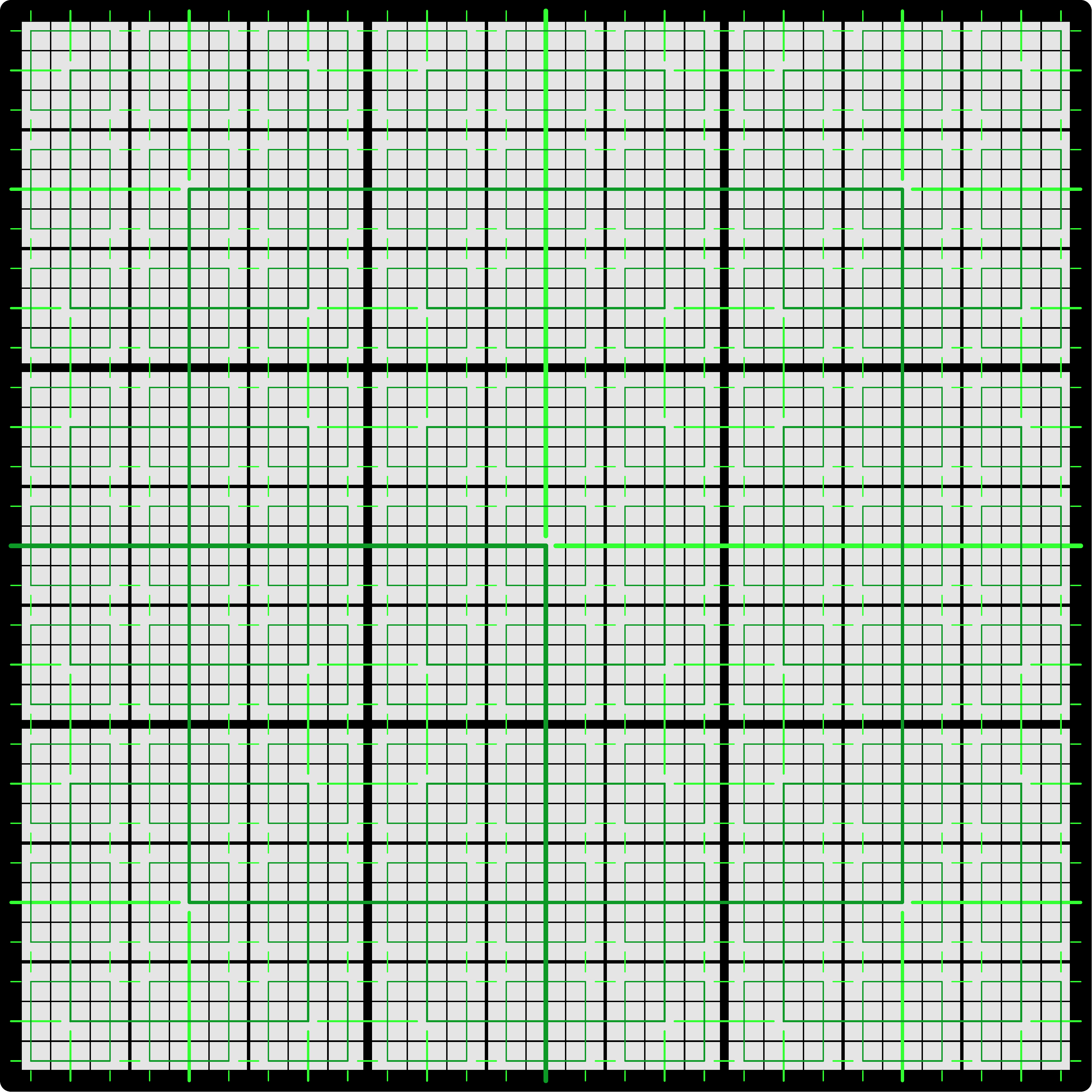}
\end{center}
\caption{The picture shows an order-3 supertile composed of 
$27\cdot 27$ tiles. Sides traversed by the same continuous green line segment have identical green indices.
Sides traversed by light green straight lines store the red indices of tiles obtained via compositions.
Sides traversed by dark green rings store the types of tiles obtained via compositions.}\label{fgr}
\end{figure}

More specifically, let $S_i$ be a supertile of order $i$ in a proper tiling.
It has 8 sequences of sides called the \emph{$i$-rays}. 
In Fig.~\ref{fgr} sides from 
the $i$-rays are crossed  by
light-green straight line segments of length $3^{i-1}$,
the length of a segment is defined as the number of sides it crosses.
The green index on all sides of each $i$-ray 
is equal to a red index of the tile $\sigma^{-i}S_i$. 
Actually, 4 $i$-rays would suffice, as each tile has 4 red indices,
but eight $i$-rays make the construction more symmetric.
Essentially the same information is stored in green indices 
of sides crossed by the dark green \emph{$i$-ring}:
the green index of all those sides is  equal
to the type of $\sigma^{-i}S_i$. The dark green
$i$-ring consists of four   straight line segments,
each of length $2\cdot 3^{i-1}$. 
For each $i$-ray, 
the dark green $i$-ring crosses a side that belongs to the $i$-ray. This 
allows to make  the type
of $\sigma^{-i}S_i$ consistent 
with red indices of $\sigma^{-i}S_i$. Each $i$-ray
is continued by a similar $i$-ray of the adjacent 
supertile, which allows to make each red index 
of $\sigma^{-i}S_i$ match the red index on 
the adjacent side.

It may seem that this is enough, and blue indices are not needed.
However, this is not the case. 
The problem is that sides whose green indices carry global
information have no information
about the type of  tiles they belong to. This may result
in that the composition of a proper tiling contains a tile
whose green contour is inconsistent with its  red 
contour.
For simple substitutions where each tile is replaced by a $3\times3$
grid this cannot happen. However, for more complex ones,
for instance, 
for the substitutions
where square prototiles are replaced by the $5\times5$ grid,
it can.

To overcome this problem we have to add more information in green indices.
It is to convenient to consider this extra information as another index called the 
``blue index''.
Almost each blue index of a tile identifies its type.
The only exceptions from this rule are blue
indices on the sides crossed by blue straight line segments
of non-unit length in
Fig.~\ref{fgr1}, where blue indices carry the information 
about blue indices of tiles obtained via compositions.
\begin{figure}
\begin{center}
\includegraphics[scale=.8]{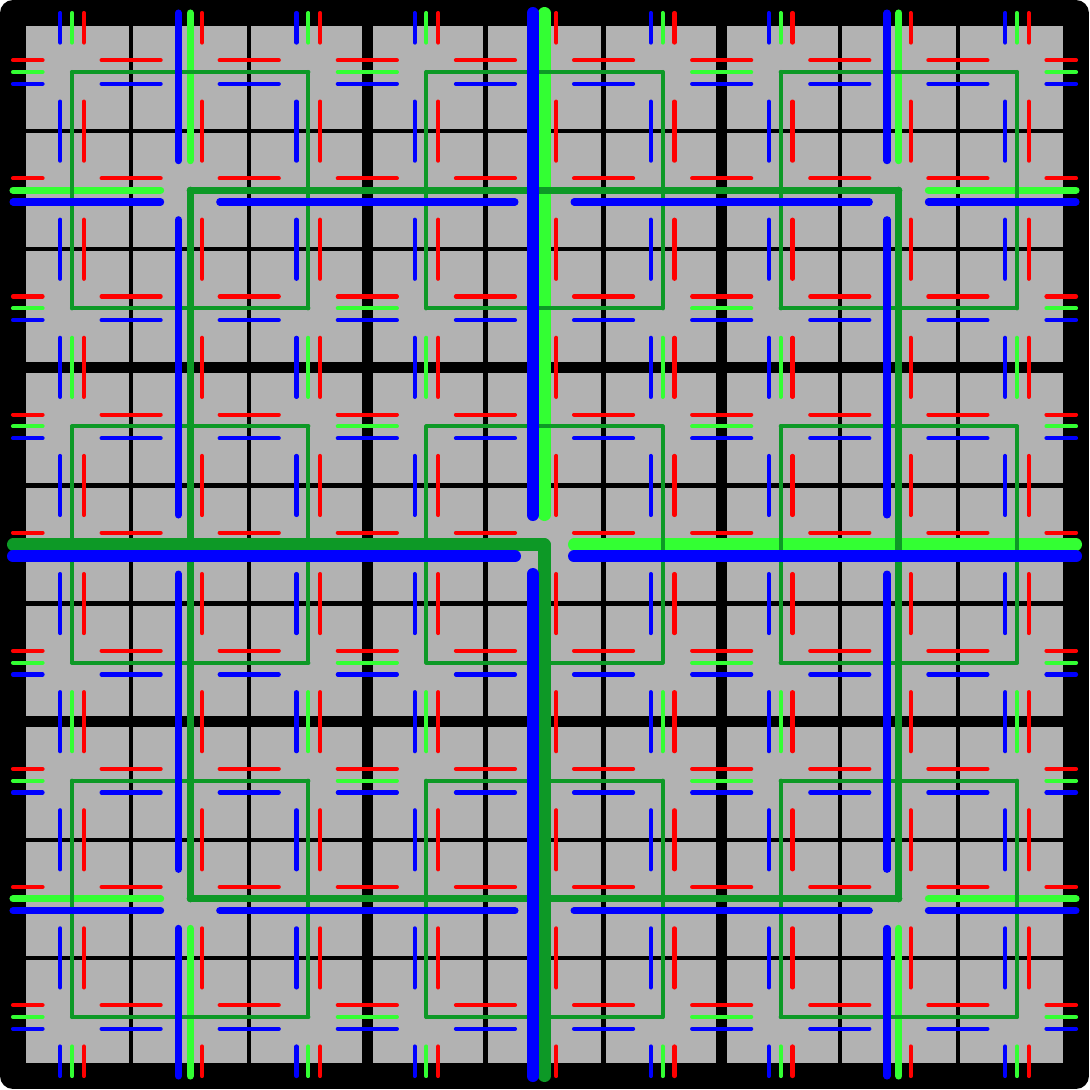}
\end{center}
\caption{The picture shows an order-2 supertile composed of 
$9\cdot 9$ tiles. 
Sides traversed by every blue line segment have the same blue indices.
}\label{fgr1}
\end{figure}
Fortunately, those exceptions do not ruin the construction.

\subsubsection{Tile decoration: the formal definition}
\label{s43}

\newcommand{\cld}{\text{Ch}}

Now let us move on to constructing
the set of decorated tiles. First,
we define the substitution $\sigma'$ on decorated tiles.
In general, it will be non-deterministic due to the fact that
the secondary ports can be decorated in several ways.

If  a decorated tile $A$ is obtained from a prototile
$\alpha$ by choosing indices,
we will say that $A$ is \emph{of the form} $\alpha$.
For each decorated tile
$A$ of the form $\alpha_i$ and for each type
$s\in \mathcal{M}_i$, we
define a set of decorated tiles $\cld_s(A)$, whose members are called
\emph{the children of $A$ of type $s$}.
The tiles of the sets $\cld_s(A)$ for all $s\in \mathcal{M}_i$ will be used to
construct  decompositions of $A$ under the substitution $\sigma'$.
If the type $s$ has no secondary ports,
then this set will be a singleton, otherwise it may
contain several tiles due to the possibility
of choosing blue indices.

Here is how the set $\cld_s(A)$ is defined (see Fig.~\ref{f1}):
\begin{itemize}
\item
  If $B\in\cld_s(A)$ then the central index of $B$ is $s$, in the sequel this index is called \emph{the type of} $B$.

The indices on the side $b$ of a tile $B$ from $\cld_s(A)$ are defined as follows:

\item Red indices:
\begin{itemize}
\item \label{redind} If $b$ is an inner side,
  then denote by $r$ the type of the
  tile from $\mathcal{M}_i$ that lies in $\mathcal{M}_i$ on the other side of $b$.
Then the red index of $B$ on $b$
is equal to the ordered pair $\pair{s,u}$ or $\pair{u,s}$
depending on whether $s$ is to the left or right of $b$.
If the side $b$ is horizontal, then the upper tile is considered to be the left one.

\item
The red index on each outer side $b$
of a tile $B$ from macrotile $\tau^m A$ is defined as the tuple  $\pair{n_1,\dots, n_m}$ 
where 
$n_i$ is the number of the side in $\tau^{i} A$, from left to right 
along the superside $\tau^{i} a$, which $b$ belongs to. 
In Fig.~\ref{f-8}  we have shown red indices on the 
north side  of a $\tau^{2}$-macrotile, where $\tau$ is the substitution of our first example (a square is substituted with a 3 by 3 grid).
\begin{figure}
  \begin{center}
    \includegraphics[scale=1]{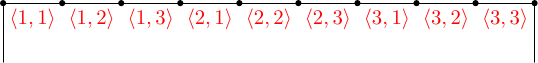}
  \end{center}
  \caption{Red indices on the north side of a $\tau^{2}$-macrotile, where $\tau$ is the substitution of our first example (a square is substituted with a 3 by 3 grid). }\label{f-8}
\end{figure}
\end{itemize}
We call a decorated tile $B$ \emph{normal} if its red indices are determined by its type according to this rule.
For example, all tiles on the right in Fig.~\ref{f1} are normal.

\item Green indices:
\begin{itemize}
\item
If $b$ belongs to the path $P_z$, where $z\in\{\text{north, east, south, west}\}$, then the green index of tile $B$ on side
$b$ is equal to the green index on $z$th side of $A$.
\item
If $b$ is a border side, then the green index of tile $B$ on side $b$ is equal to the type of $A$.
Therefore, we will call this index of $B$ \emph{the parent index of} $B$.
\item
If $b$ is a secondary port on the macroside $\sigma a$,
then the green index of $B$ on side $b$
is equal to the \emph{red} index on side $a$ of $A$.
\item
In the remaining case  $b$ is an inner non-net side. Then
its green index is zero.
\end{itemize}

\item Blue indices:
\begin{itemize}
\item
Similarly to green indices, if $b$ belongs to the path $P_z$, then the blue
index of tile $B$ on side
$b$ is equal to the blue index on $z$th side of $A$.
\item

If $b$ is a non-outer non-net side and is $z$th side if $B$,
then the blue index on it is the quadruple $\pair{s,z,r,u}$ 
or $\pair{r,u,s,z}$ depending on whether $s$ is to the left or right of this side.
Here $r$ denotes  the type of the
tile $C$ from $\mathcal{M}_i$ that lies on the other side of side $b$
and $u$ denotes its name in $C$.

\item 
Finally, if $b$ is a secondary port and is $z$th side of $B$, then similarly to the previous case the blue index on side $b$ 
can be any admissible quadruple $\pair{s,z,r,u}$ 
or $\pair{r,u,s,z}$ depending on whether $s$ is to the left or right of this side.
 We call a quadruple $\pair{r,u,s,z}$ \emph{admissible}
if $r$ is to the left of its $u$th side and 
there is a shift $r'$ of $r$ that 
does not overlap with  $s$ and such that $z$th side of $s$ coincides with $u$th side of $r'$. Note 
that blue indices on non-outer non-net sides are admissible quadruples as well.
Since $r$ can be chosen in several ways, the set $\cld_s(A)$ can have several tiles.
\end{itemize}
\end{itemize}

\begin{definition}
If $B\in\cld_s(A)$ then we call $A$ a \emph{parent} of $B$. 
\end{definition}
\begin{remark}
A tile can have several parents or no parents at all. In the former case all its parents
are of the same form. If $B$ is a border tile, then, moreover, all its parents
are of the same type.
All tiles that have a parent are normal.
\end{remark} 

\begin{definition}
If $B\in\cld_s(A)$ then
blue and green indices of net sides, green indices of border sides
and of secondary ports in tile $B$
are borrowed from $A$. These indices
are called \emph{borrowed},
and the corresponding index of $A$ is called the \emph{source} of that borrowed index. Non-borrowed indices are called \emph{native}.
\end{definition}
It follows from the definition that each index on each side $a$ of a
decorated tile $A$
is borrowed by some outer side of some tile in $\cld_s(A)$
for some $s$ depending on the type of tile $A$, on $a$ and the color of the index.

\begin{definition}
Let a decorated tile $A$ of the form $\alpha_i$ be given.
For each type $s$ from $\mathcal{M}_i$,  choose some tile $B_s$ from $\cld_s(A)$ and replace
in $\mathcal{M}_i$ the type $s$ with the tile $B_s$. Then
we obtain a proper tiling. 
Tilings obtained in this way are called
\emph{decompositions of $A$ under the substitution} $\sigma'$.
\end{definition}

For the first substitution, a decorated macrotile is shown
in Fig.~\ref{f1} on the right, and the tile $A$ itself is drawn on the left.
In this example, we could make the 
substitution $\sigma'$ deterministic, that is, make all sets $\cld_s(A)$
singletons.
This is because we have only one prototile,
so we know which macrotile should be the neighbor on each side. However, in general, we cannot make all $\cld_s(A)$ singletons, and we
need non-deterministic substitution.

\begin{obs}\label{rem2}
Let a type $s$ from a $\sigma$-macrotile $\mathcal{M}_i$ be fixed. 
Then the set $\cld_s(A)$ 
depends only on the blue-green\footnote{The blue-green index on a side of a tile
is defined as the pair consisting of
its blue and green indices.} index
on $z$th side of  $A$ provided the net path $P_z$ passes through $s$
and on the type of 
$A$ and its red indices provided 
$s$ is a border type. 
\end{obs}
\begin{proof}
Indeed, to compute tiles $B$ from $\cld_s(A)$, in addition to $s$, only the borrowed indices of $B$ are needed. 
If $s$ is not a border tile, then only the indices on net sides are borrowed. 
Otherwise, the type and red indices of tile $A$ are also borrowed, 
which become the parent index and green indices of secondary ports of $B$. 
\end{proof}

\subsubsection{Legal Tiles}\label{slt}

Now we will define our set of decorated tiles by imposing some
constraints on the decoration of the
 prototiles. Tiles that satisfy 
those constraints will be called
\emph{legal}.
For example, we want the red indices to force
decorated tiles to assemble into macrotiles, so
all tiles in our set must be normal. 
 For each specific initial substitution, we can define legal tiles explicitly,
but since we want our construction to be general, we will use
a different approach.

It is clear that any legal tile must have a parent, and moreover, a
legal parent.
Indeed, a proper tiling with legal tiles
must be composable, so any legal tile must be included in some
decorated macrotile whose composition is legal.
Therefore, we need to remove all decorated tiles that have no parents.
After that, we will have to remove
tiles whose parents were all removed.
And so on. Each new removal can increase the number of
tiles all of whose parents were removed. More or less, we will call a decorated tile legal if it has a parent, which in turn
has a parent, and so on, infinitely many times.

But unfortunately this is not enough. For example, for the first 
substitution, any normal
 type-\texttt e tile is its own parent.
So it will never be removed. But there are too many such
tiles, and they yield parasite tilings.
So we will impose another requirement.

\begin{definition}
We call a decorated tile $A$ \emph{legal} if
there is an infinite sequence of decorated tiles
$A_0=A,A_1,A_2,\dots$ with the following properties:
\begin{itemize}
\item $A_{i+1}$ is a parent of $A_{i}$ for all $i$,
\item
Consider some borrowed index $I_i$
in $A_{i}$ and its source $I_{i+1}$ in $A_{i+1}$. The latter can also
be borrowed, then we consider its source  $I_{i+2}$ in $A_{i+2}$ and so on.
If the sequence $I_{i}, I_{i+1},I_{i+2}, \dots$ ends at a native index,
then we call the latter \emph{the origin} of all indices in the sequence.
Otherwise, we say that the original index \emph{has no origin}.
It is required that for all $i$ all indices in $A_{i}$ without an
origin are zero.

\end{itemize}
\end{definition}

\subsubsection{Properties of legal tiles}
To get used to legal tiles,
let us formulate some of their simple properties:

\renewcommand{\labelenumi}{S\theenumi:}
\begin{enumerate}
\item Every legal tile is normal.
\item Every legal tile has a legal parent --- this is the second term
$A_1$  of the
 ancestor sequence that witnesses legality.
\label{s1}
\item Every child of every legal tile is legal as well.\label{s3}
\item Green indices on different border sides of a legal tile
coincide.
\label{s2}
\item Any blue index of a legal tile is either an admissible quadruple or zero.\label{s5}
\item Blue-green indices on net sides of a legal tile coincide.
\label{s6}
\end{enumerate}

The most important
property of legal tiles is the so-called
\emph{Dichotomy Lemma}. Reading its rather complicated statement and proof
can be postponed until it becomes clear why it is needed.

To state the lemma, we need some new notions.
Let tiles $A$ and $B$ have the same form.
We say that   $A,B$ are  \emph{similar on side }$z\in\{\text{north, east, south, west}\}$ if $A$ and $B$ 
have the same blue-green index on $z$th side.
We say that $A,B$ are \emph{similar} if they are similar on all their  sides.

\begin{lemma}[Dichotomy Lemma]\label{l10}
Let a legal tile $D$ and a central cyclic or non-central type $t$
of the same form as $D$ be given.
Then the following dichotomy holds:
either $D$ is similar to some legal tile of type $t$,
or there is $z\in\{\text{north, east, south, west}\}$ such that
$D$ is not similar to \emph{any} legal tile of type $t$ on $z$th side.
{\rm Reformulation:} if for each $z$  the tile $D$ is similar to
some legal tile $C_z$ of type $t$ on $z$th side, then $D$ is similar to 
some legal tile $C$ of type $t$ (on all sides).
\end{lemma}


\begin{proof}
Assume first that $t$ is a central cyclic type. We claim that then the first alternative holds.
As $t$ is a cyclic type, it is its own central child.
Since $D$ and  $t$ have the same form, their central children are of the same type,
      \begin{center}
        \includegraphics[scale=1]{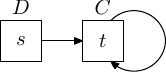}
      \end{center}
thus the central child $C$ of $D$ is of type $t$. 
The 
blue-green contour of $C$ is borrowed from  $D$.  The tile $C$ is legal because it is a child of 
the legal tile $D$; recall property S\ref{s3}. 

Now assume that $t$ is a non-central type.
Let 
$$
\dots\to D_2\to D_1\to D_0=D
$$ 
be a sequence of tiles that witnesses legality of $D$. 
Consider two cases.  

\emph{Case 1:} some tile $D_k$ has a non-central type, say type $s$.
Consider the minimal  such $k$. 
If $s=t$, then the first alternative holds, since the blue and green
indices of the tile $D_k$
pass unchanged to $D$.
We claim  that otherwise the second option holds. To prove the claim consider two cases.

\emph{Case 1a:} Assume first that  there is $z\in\{\text{north, east, south, west}\}$
such that $z$th side is not in the net in both types $s$ and $t$. Then the statement
follows from the following
\begin{obs}\label{o-diff1}
 If  $z$th side is not in the net in both types $s$ and $t$ and $s\ne t$,
then 
the blue index on $z$th side of any legal type-$t$ tile 
is different from that of any legal type-$s$ tile.
\end{obs}
\begin{proof} 
Let $y_s$ and $y_t$ denote those blue indices. 
W.l.o.g. assume that $s$ is to the left of its $z$th side.
Distinguish   the following cases:
\begin{itemize} 
\item $t$ is to the left of its $z$th side. Then   $y_s=\pair{s,z,*,*}$ and $y_t=\pair{t,z,*,*}$ hence $y_s\ne y_t$.
\item $t$ is to the right of its $z$th side. Then $y_s\ne y_t$ unless
$y_s=y_t=\pair{s,z,t,z}$. And the latter is impossible, since no admissible quadruple has the form  $\pair{s,z,t,z}$.
\end{itemize}
Note that in this proof we used only once that the tiles $t,s$ have rectangular form.
That assumption was used to prove that \emph{no admissible quadruple has the form  $\pair{s,z,t,z}$}.
\end{proof}

\emph{Case 1b:} Otherwise there is no $z\in\{\text{north, east, south, west}\}$
such that $z$th side is not in the net in both types $s$ and $t$. 
As types $s,t$ are non-central, they both  have at most 
two net sides. Hence $s$ has two net sides such that  both eponymous
sides in $t$ 
are not in the net. Let $a,b$ denote their names. Blue indices on $a$th and $b$th sides 
of $D_k$ (and hence of $D$) coincide. Denote them by $y$.
The statement
follows from the following

\begin{obs}\label{o-diff2}
Assume that $a\ne b$, both $a$th and $b$th sides are non-net sides in type $t$
and $y$ is a blue index of  a legal tile. 
Then there is a side $z\in\{a,b\}$ such that $y$  is different from the blue index on $z$th side 
 of any 
legal tile of type $t$.
\end{obs}
\begin{proof}
Every blue index of a legal tile is either zero, or an 
admissible quadruple (property~S\ref{s5} of legal tiles).
If $y=0$ then the blue index on both $a$th and $b$th sides of any type-$t$
tile is different from $y$ and we can let $z=a$ or $z=b$.  
Otherwise  $y$ is an admissible quadruple $\pair{u,c,v,d}$. 

W.l.o.g. assume that $t$ is to the left of its $a$th side. Distinguish now the following cases:
\begin{itemize}
\item $c\ne a$ or $u\ne t$; then we can let $z=a$,
\item $c=a,u=t$ and $t$ is to the left of its $b$th side; then we can let $z=b$, since $c=a\ne b$.\footnote{Actually, this case cannot happen, as $a$ and $b$ being non-net sides are parallel. We analyze this case, since in the sequel we
will consider non rectangular tiles and we want the argument be valid also
in that case.}
\item $t$ is to the right of its $b$th side and $d\ne b$ or $v\ne t$; then we can let $z=b$.

In the remaining case
$t$ is to the right of its $b$th side and $y=\pair{t,a,t,b}$. Since $y$ is an admissible quadruple,
there is a shift $t'$ of $t$ whose $a$th side coincide with $b$th side of $t$.
Thus $a$th and $b$th sides of $t$ are parallel.
At least one of them is a non-outer side.
\item If $a$ is a non-outer side then we can let $z=a$, as 
it cannot happen that the blue index on $a$th side of a type-$t$ tile $B$ has the form $\pair{*,*,t,b}$ ---
to the right of $B$ there is a tile of type different from $t$.
\item Otherwise $b$ is a non-outer side and we can  let $z=b$.\qed
\end{itemize}
\renewcommand{\qed}{}
\end{proof}

Later we will need this observation for non-rectangular tiles.
Note that in its proof we used only once that the tile $t$ has rectangular form.
That assumption was used to establish the following property:
\begin{enumerate}
\item[]``Outer Sides Requirement'': \emph{No type $t$ has  parallel \emph{outer} sides $a,b$ 
\label{outersides} such that $t$ is on the left of $a$ and on the right of $b$}.
\end{enumerate}

\emph{Case 2:} all $D_i$ are of central type. We claim that then the second alternative holds.
Indeed, in this case all blue indices of $D$ are zero. On the other hand,
since $t$ is a non-central type, it has a non-net side, which thus carries a non-zero blue index.
%
\end{proof}

\subsubsection{Proof of the correctness of the construction}
We need to prove that a tiling with undecorated tiles
is hierarchical if and only if its tiles can be decorated so that
the result is a proper tiling with legal tiles.
First, we prove the easy direction: it is possible
to properly color any hierarchical tiling.

\begin{prop}\label{th1}
Every $\tau$-hierarchical  tiling with undecorated tiles is a projection of
a proper tiling with legal tiles.
\end{prop}
\begin{proof}
  Fix a hierarchical tiling $\mathcal T$ and an infinite sequence $\mathcal T_0=\mathcal T, \mathcal T_1,\mathcal T_2,\dots$,
  in which each tiling is a $\tau$-composition of the previous one. 
  Consider its subsequence $\mathcal T_0, \mathcal T_m,\mathcal T_{2m},\dots$.
  In each tiling  $\mathcal T_{im}$ group the tiles into $\sigma$-macrotiles according
  to $\mathcal T_{(i+1)m}$. Since $\mathcal T_{(i+1)m}$ is side-to-side,
  in $\mathcal T_{im}$ every $\sigma$-macroside is adjacent to a $\sigma$-macroside. 
 By the choice of the location of central tiles, each main port is adjacent to a main port in $\mathcal T_{im}$.
  
  We want to properly color all tilings of our sequence so that for the decorated tilings
  $\mathcal T'_0, \mathcal T'_{m},\mathcal T'_{2m},\dots$, each tiling is a $\sigma'$-composition of the previous one. 
  To do this, we color each macrotile $\mathcal {M}$ from
the tiling $\mathcal T_{im}$ as described in Section~\ref{s43}. We choose the undefined components of blue indices of secondary ports so that they are
  the same for adjacent tiles. 
  The borrowed indices of the tiles from $\mathcal {M}$ are yet undefined since we have not yet completely decorated the parent $D$ of $\mathcal {M}$, which belongs to $\mathcal T_{(i+1)m}$.
  We set the borrowed indices that have an origin to their origin, and set the borrowed indices without an origin to zero.
  
By construction, this decoration has the following properties.
 
  (a) \emph{All tilings of the chain $\mathcal T'_0, \mathcal T'_m,\mathcal T'_{2m},\dots$ are proper.}
  Indeed, all indices on non-outer sides of each macrotile are the same on adjacent tiles by construction. We claim that
  native indices on outer sides are also the same. Indeed, 
  red indices on outer sides of tiles from $\mathcal T_{im}$ coincide, since all the tilings $\mathcal T_{im+1},\dots,\mathcal T_{(i+1)m}$ are side-to-side.
  Native blue indices coincide because of the choice of their undetermined components. 
   Borrowed indices on the outer sides  that have an origin
  are equal to that origin. If two outer sides from different macrosides are  adjacent, then their sources are adjacent as well.
  Therefore, their origins are adjacent. And at the origin, the indexes are native, so they coincide.
  Finally, borrowed indices without an origin are equal to zero, so they coincide.
 
  (b) \emph{All the resulting decorated tiles are legal.} Indeed, it follows from the construction that after decoration,
  each macrotile is a $\sigma'$-decomposition of its parent. Consider an arbitrary decorated tile $A$ from any of the tilings of the chain
  $\mathcal T'_0, \mathcal T'_{m},\mathcal T'_{2m},\dots$.
  Consider its ancestors
  $$
  \dots\to A_{3m}\to A_{2m}\to A_m\to A_0=A
  $$
  By construction, all indices without an origin in tilings from this sequence are zero. Therefore, the tile $A$ is legal.
\end{proof}

Let us now prove that any proper tiling with legal tiles has a $\sigma'$-composition
that is proper and consists of legal tiles only.

\begin{prop} \label{th2}
  Every proper tiling $\mathcal T$ with legal tiles has a unique proper $\sigma'$-composition $\mathcal T'$ consisting of legal tiles.
  The tiling  $\mathcal T'$ has the following feature:
all tilings $\tau^{1} \pi(\mathcal T'), \dots, \tau^{m-1} \pi(\mathcal T')$
are side-to-side.
\end{prop}
\begin{proof}
The key lemma in the proof is the following 
\begin{lemma}[Composition Lemma]\label{l9}
Let $\mathcal S$ be a finite proper tiling 
that is equal to $i$th $\sigma$-macrotile 
$\mathcal M_i$  provided we ignore red, blue and green indices. 
Assume that all tiles in $\mathcal S$ are legal, except possibly for the  central tile $A$. 
    Then the following hold: 
    \begin{enumerate}
    \item[(a)] All border tiles of  $\mathcal S$ have the same parent index $t$, which is of the form $\alpha_i$.
    \item[(b)] The tiling $\mathcal S$
      has the unique $\sigma'$-composition $D$.
      The tile $D$ has type $t$ and is normal.
    \item[(c)] 
For all  $z\in\{\text{north, east, south, west}\}$
the blue-green index on $z$th side of $D$
is equal to that of some legal tile $C_z$ of type $t$.
  
    \item[(d)] If the central tile $A$ is legal and  $t$ is a non-central or central cyclic type, then $D$ is legal as well.
    \end{enumerate}
  \end{lemma}
  \begin{proof}
    (a) Due to property~S\ref{s2}, all the border tiles of
$\mathcal{S}$ have the same parent index $t$. To prove that $t$ is of the form $\alpha_i$, 
it suffices to find a legal type-$t$ tile of the form $\alpha_i$.
   Such a tile is any legal parent $C$ of any border tile $B$ in $\mathcal S$.
   The form of $C$ is $\alpha_i$, since $C$ has a child in $\mathcal S$ and $\pi(\mathcal S)=\mathcal M_i$. And the type of $C$ is $t$, since so is
 the parent index of $B$.

(b) The uniqueness of such a tile $D$ is obvious:
 its type is uniquely determined by the parent index of tiles in $\mathcal S$ and
 its blue-green contour is uniquely determined by blue-green indices of main ports of $\mathcal S$.
 Finally, its red contour  is determined by green indices of secondary ports in $\mathcal S$.
 More specifically, in theory it could happen that different secondary ports 
 on the same macroside of $\mathcal S$ have different green indices.
 It will follow from our arguments that this cannot happen. But for the time being 
let us choose for each side $d$ of $D$ any tile  $E_d\in\mathcal S$ that has a secondary 
port $p$ on the macroside $\sigma d$ and let the red index of $D$ on $d$ be equal to the green
index of $E_d$ on the side $p$. 

Let us show that this tile $D$ is normal. We have to find, for each its side $d$, 
a normal tile of type $t$ with the same red index  on $d$.
This is any legal parent $F_d$ of the chosen tile $E_d$. Indeed, 
 its  type is $t$, since so is the parent index of $E_d$. 
Besides,
\begin{align*}
&\text{red index of }F_d \text{ on }d 
= \text{ green index of  }E_d \text{ 
   on  }p\\ 
=&\text{ red index of }D \text{ on }d.
\end{align*}

Let us show 
that $\mathcal S$ is a decomposition of $D$,
that is, for all types $s$ in $\mathcal M_i$, the tile $B$ of type $s$ in $\mathcal{S}$ 
is in $\cld_s(D)$.

Assume first that $s$ is a non-central type. We know that $B$ is legal, and hence for some tile $C$ of the form 
$\alpha_i$ it is in $\cld_s(C)$. We claim that $\cld_s(D)=\cld_s(C)$.
By Observation~\ref{rem2}, to prove the claim, it suffices to prove two statements:\\
    (1) $C$ and $D$ have the same blue-green index on $z$th side if the path 
    $P_z$ contains $s$, and\\
    (2) $C$ and $D$ have the same type and the same red indices, if $s$ is a border type.\\
    The  statement (1) follows from the following chain of 
equalities for blue-green indices:
\begin{align*}
&\text{the index of }C \text{ on }z\text{th side} \\
= &\text{ the index on any side in the path  }P_z \text{ 
   in the decomposition of }C \\
= &\text{ the index of }B \text{ on any side of the path }P_z \text{ in }\mathcal S\\
=&\text{ the  index of }D \text{ on }z\text{th side}.
\end{align*}
Here the first equality holds by the definition of the substitution $\sigma'$,
the second one holds since $B$ belongs to the decomposition of $C$,
and the last one holds by construction of $D$.

 The first part of statement (2) follows from the fact that the type of $D$ is $t$ by construction,
 and the type of $C$ is $t$ since it is a parent of $B$, which has parent index $t$. 
The second part follows from the fact that both $C$ and $D$ are normal. 
    
   It remains to show that $A$, the central tile in $\mathcal S$, is in  $\cld_s(D)$, where
   $s$ is the central type in $\mathcal S$. 
     We have to prove that on each side $a$ of $A$ all
    three indices are as required by the definition of $\cld_s(D)$. Choose any such side $a$.
    On this side, $A$ has the same indices as its  neighbor $B$ on side $a$,
    since the tiling $\mathcal{S}$ is proper.
    And $B$  has the indices required for the child of $D$ of its type.
    By construction,  the indices of adjacent children on shared sides coincide.
    Therefore, the tile $A$ on the side $a$ has the required indices.

    (c) 
    On the path $P_{z}$ in macrotile $\mathcal M_i$, there is a border tile $B$,
    see Fig.~\ref{f8}.
    \begin{figure}
      \begin{center}
        \includegraphics[scale=.7]{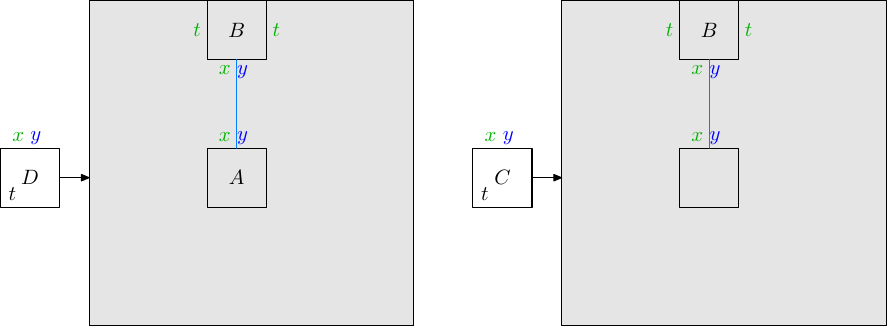}
      \end{center}
      \caption{On the left is a proper tiling $\mathcal S$ with $\pi(\mathcal S)=\mathcal M_i$ and with the  central tile $A$. Its composition is $D$.
      The net path in $\mathcal S$ containing $z$th side of $A$ (the north side in the figure)
        is shown in blue. 
        On the right is shown the legal parent $C$ of $B$ and the decomposition of $C$. 
        The type of $C$
        and the parent indices in its decomposition and in the tiling $\mathcal S$ are  denoted
        by $t$. The figure shows that tile $D$ has the same blue-green index on $z$th side 
as tile $C$ on $z$th side.}\label{f8}
    \end{figure}
    Consider any its legal parent $C$, it has type $t$, since so is the parent index of $B$. We claim that the blue-green index on $z$th side 
of $C$ is the same as that on $z$th side of $D$. This follows from the chain of equalities for blue-green indices:
    \begin{align*}
      &\text{the  index of }C \text{ on $z$th side }\\
      =&
      \text{ the index of }B \text{ on the sides of the path } P_z\\ =
      &\text{ the  index of }D \text{ on $z$th side } 
    \end{align*}
    Both equalities holds by the definition of substitution on decorated tiles.

    (d) Assume now that the central tile $A$ is legal and $t$ is either a non-central, or a central cyclic type. 
We have already shown that $D$ is normal. It remains to prove that 
its blue-green contour is good, that is, that $D$ is similar to a legal tile of type $t$. 
Informally, the central tile $A$ has verified that the  blue-green contour of $D$
is legal, that is, $D$ is similar to some legal tile. On the other hand,
border tiles from the net have verified that each blue-green index of $D$ is equal to that
of some legal type-$t$ tile. And Dichotomy Lemma guarantees
that in such circumstances  $D$ is similar to a legal tile of type $t$.

More specifically,
let $D'$ denote any legal parent of $A$. Since $D$ and $D'$ are similar, 
it suffices to prove that $D'$ is similar to some legal tile of type $t$. 
As $D'$ is a parent of $A$, for each $z$
the  blue-green index on $z$th side of   $D'$ is equal to that of $A$
and hence of $D$.
And by item (c) the latter equals to that of some legal tile $C_z$ of type $t$.
By Dichotomy Lemma
$D'$ is similar to a legal  type-$t$ tile.  
  \end{proof}

Let us continue the proof of Proposition~\ref{th2}.
Let a proper tiling $\mathcal T$ with legal tiles be given. Red indices guarantee that it can be partitioned into macrotiles,
and in the unique way. For each of these macrotiles $\mathcal{S}$ there is $i$ with $\pi(\mathcal{S})=\mathcal M_i$.
Replace each macrotile $\mathcal S$ of the original tiling by the tile $D$ existing by item (b)
of the Composition Lemma.
  We obtain a composition of $\mathcal T$.
 Red indices on outer sides ensure that the composed tiling is  side-to-side.
Each main port in $\mathcal T$ is adjacent to a main port and each secondary 
port is adjacent to a secondary port.
Since $\mathcal T$ is proper, all indices of its tiles on adjacent ports match.
And since all indices on sides of composed tiles are borrowed from ports of $\mathcal T$,
the composed tiling $\mathcal T'$ is proper. 

Let us show the tiling  $\mathcal T'$ has the required feature:
all tilings $$\tau^{1} \pi(\mathcal T'), \dots, \tau^{m-1} \pi(\mathcal T')$$
are side-to-side. 
Let $D$ be a tile from  $\mathcal T'$.  Since all $\tau$-supertiles are side-to-side tilings,
all supertiles $\tau^{m-1}\pi(D),\dots,\tau \pi(D)$ are side-to-side. So the problem may occur
only if for adjacent tiles $D,E\in  \mathcal T'$
there is $i<m$ such that some tiles $D'\in\tau^i \pi(D)$ and $E'\in\tau^i \pi(E)$
have sides $a',b'$ that share a segment of positive length but do not coincide.
This contradicts the fact that the supertiles $\tau^m D,\tau^m E \in \mathcal T$
have the same red indices on the shared superside. Indeed, this implies that $a',b'$ are obtained
from pairs of adjacent  sides of tiles from $\tau^m D,\tau^m E$ and hence coincide.

  It remains to prove that each tile $D\in\mathcal T'$  is legal.
  By item (d) of the Composition Lemma, $D$ is legal unless $D$ is of central
  acyclic type.
  We have to prove that  $D$ is legal even in this case.
  
Since all composed tiles are normal, $\mathcal T'$ can be again split into macrotiles.
Consider a macrotile $\mathcal M'\subset\mathcal T'$  containing a tile $D$ of a central acyclic  type $t$.
  Denote by $r$ the parent indices of tiles in $\mathcal M'$. By the Composition Lemma (b), the tiling $\mathcal M'$ has a composition $C$ of type $r$,
see Fig.~\ref{f40}.
  \begin{figure}
    \begin{center}
      \includegraphics[scale=1]{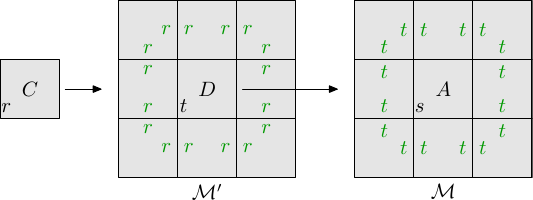}
    \end{center}
    \caption{Tiles $C,D,A$ have  types are $r,t,s$, respectively, where $r$ is a non-central type and $t$
is a central acyclic type.
      The macrotile $\mathcal M'$ is a decomposition of $C$, and the macrotile $\mathcal M$ is a decomposition of $D$.
      The figure illustrates the proof that  $D$ is legal.}\label{f40}
  \end{figure}
   
Since the central child of $r$ is an acyclic type $t$,
by Cyclicity Requirement $r$ itself is not central.
By item (c) of the Composition Lemma, for each $z\in\{\text{north, east, south, west}\}$  
there is a legal tile $F_z$ of type $r$ whose  
blue-green index on $z$th side coincides with that of $D$.
 Hence the same holds for $C$:   for each side $C$  is similar to 
a legal tile of type $r$ on that side.

We would like to apply Dichotomy Lemma to $C$,
but we may not do that, as we have not proved yet that $C$ is legal.
Therefore let  $D'$ denote any legal parent of $A$, and $C'$ any legal parent of $D'$. 
Then $D'$ 
and $D$ have the same blue-green contour, thus so have $C$ and $C'$. 
Hence, like $C$,  for each $z$ the tile
$C'$ is similar to a legal tile $F_z$ of type $r$ on $z$th side.
By the Dichotomy Lemma, $C'$ is similar to a legal tile $F$ of type $r$.
So is $C$, since it has the same blue-green contour as $C'$.
This implies that the tile $C$ is legal.
Therefore, $D$ is legal, being a child of a legal tile $C$.
\end{proof}

 Theorems~\ref{th0}(a) and~\ref{th22}(a) are proved.
 
\subsection{Proof of Theorem~\ref{th0}(b)}
We first show that every substitution tiling is hierarchical. In this proof 
we will not use the assumption that all tiles are 
squares. Instead we will assume that 
all supertiles are side-to-side tilings.

\begin{lemma}\label{th11}
Assume that 
all supertiles are side-to-side tilings. Then
every substitution tiling $\mathcal{T}$ is hierarchical.
Moreover, every substitution tiling has a composition that is
also a substitution tiling.
\end{lemma}
The converse is not true in general.
\begin{proof}
As  all supertiles are side-to-side tilings,
all substitution tilings are side-to-side.
Thus the first statement of the theorem follows from the second one.

Let us prove the second statement.
Let $\sigma$ be a substitution and $\mathcal{T}$ a substitution tiling.
Call any supertile $S$ for which
$F\subset\sigma S$
a \emph{cover} of a tiling $F$.
Any finite $F\subset \mathcal T$ has a cover. Indeed,
$F$ is included in some supertile and
the composition of that supertile is a cover of $F$.

Let
$A_1,A_2,\dots$  be an enumeration of all tiles from $\mathcal T$.  For each $n$
let $S_n$ denote  any cover  of the set $\{A_1,\dots, A_n\}$.
For $i\le n$, we call the tile $B\in S_n$ for which $A_i\in \sigma B$  the \emph{parent of $A_i$ in $S_n$}. For $m>n\ge i$, the tile
$A_i$ may have different parents in $S_m$ and in $S_n$.
However, by removing some terms from the sequence $S_1,S_2,\dots$, we can
ensure that this does not happen, namely, that for all $m> n$
the parents of $A_n$ in $S_m$ and $S_n$ coincide.

This is done using a diagonal construction. First, note that
for any tile $A_i$, the set of all possible parents of $A_i$ is finite. Indeed, a parent of $A_i$
can be identified by its form and the location of  $A_i$ in its decomposition.

Now we choose any tile $B_1$ such that for 
infinitely many $i$ the tile $B_1$ is the parent  of $A_1$  in
$S_i$. Remove from the sequence
$S_ 1,S_2,\dots$ all tilings $S_i$ for which the parent of $A_1$ in
$S_i$ is different from $B_1$. Denote by
$S'_1,S'_2,\dots$
the resulting infinite sequence of tilings.
Fix the first member $S_1'$ in it, and
thin out the sequence $S'_2,S'_3,\dots$ so that $A_2$ has
the same parent in all its tilings.
Denote by
$S''_2,S''_3,\dots$
the resulting infinite sequence. And so on. The sought sequence is $S'_1,S''_2,S'''_3,\dots$

So, we can
assume that for all $m> n$
the parents of $A_n$ in $S_m$ and $S_n$ coincide.
Now we can construct a composition of the tiling $\mathcal T$.
This is the set
$$
\mathcal T'=\{\text{the parent of }A_n\text{ in }S_n\mid n\in\N\}.
$$
By construction, the decomposition of this set contains
all the tiles $A_1,A_2,\dots$. It remains for us to prove
that $\mathcal T'$ is a substitution tiling,
in particular, different tiles from $\mathcal T'$ do not overlap.

It suffices to prove that for all $n$ the set
of tiles
$$
\{\text{the parent of }A_1 \text{ in }S_1,
\text{the parent of }A_2 \text{ in }S_2,\dots,
\text{the parent of }A_n \text{ in }S_n\}
$$
is included in some supertile.
By construction, this set
coincides with the set
$$
\{\text{the parent of }A_1 \text{ in }S_n,
\text{the parent of }A_2 \text{ in }S_n,\dots,
\text{the parent of }A_n \text{ in }S_n\},
$$
which is included in the supertile $S_n$.
\end{proof}

Let us prove now Theorem~\ref{th0}(b). Let $\tau$ denote the given substitution.
We first introduce the notion of a crown.
\begin{definition} 
Let
$\mathcal{T}$ be a tiling  with $\tau$-tiles.
  \emph{A crown} of  $\mathcal{T}$  at a vertex $V$ of its tile is the fragment
  of $\mathcal{T}$ consisting of all tiles
  from $\mathcal{T}$ that include $V$. 
  A crown is called \emph{allowed} if it is a crown at some interior vertex of some $\tau$-supertile.
 \end{definition}

Assume that for the given substitution $\tau$ 
all prototiles are squares of the same size.
We have to show that
the family of \emph{substitution} tilings for $\tau$ is sofic.
Again, we first choose a power $\sigma=\tau^m$ satisfying Lemma~\ref{l6}.

Obviously, all crowns in any $\tau$-substitution  tiling are allowed.
It turns out that for hierarchical tilings the converse is also true in a sense:

\begin{lemma}\label{l22}
  Assume that a sequence of tilings
  $\mathcal{T}_0=\mathcal{T}, \mathcal{T}_m, \mathcal{T}_{2m}, \dots$
  witnesses that the tiling   $\mathcal{T}$ is $\tau^m$-hierarchical. 
  Assume further that 
  all crowns in all tilings
  $\mathcal{T}_{im}$ are allowed. 
  Then $\mathcal{T}$ is a $\tau$-substitution tiling. 
\end{lemma}
\begin{proof}
  Let $F_0$ be an arbitrary finite fragment of $\mathcal{T}$. Let us prove that it is included in some supertile.
  Consider all tiles in $\mathcal{T}_m$ whose $\tau^m$-decompositions  intersect $F_0$. Denote this fragment by $F_1$.
  Consider all tiles in $\mathcal{T}_{2m}$ whose decompositions intersect $F_1$.
  Denote this fragment by $F_2$. And so on.
  For sufficiently large $k$, the fragment $F_k$ consists of only one tile,
  or two tiles with a common vertex, or three tiles with a common vertex, and so on.
  That is, $F_k$ is covered by one crown of the tiling $\mathcal{T}_{km}$ 
  and by assumption
this crown is allowed, that is, it is contained in some supertile $S$.
  It follows that $F_0$ is contained in the $km$-fold decomposition of the supertile $S$,
  which is also a supertile.
\end{proof}
We will rely on this lemma when constructing our set of tiles.
Note that the total number of crowns that are side-to-side tilings is finite.
Let us number them.

We then construct the set of legal tiles as before so that Propositions~\ref{th1} and~\ref{th2} hold, 
but this time we add some extra information to blue indices.
As a result, tiling the plane will become more difficult,
so the family of proper tilings will decrease or remain the same.

The definition of the set $\cld_s(A)$ does not change,
except for the definition of blue indices on non-inner non-net sides.
Namely, let $s$ be a type from $i$th macrotile $\mathcal M_i$ for $\sigma=\tau^m$ and $a$ its non-net side that   has an endpoint $V$
lying on the boundary of $\mathcal M_i$. Then the blue index on $a$ of every tile $B\in\cld_s(A)$
is supplemented with the number of any allowed crown 
that is consistent with $s$, that is, containing tile $s$ in the corresponding place.
Moreover, if it happens that the type $s$ has two non-net sides $a,b$ with a common vertex $V$
lying on the boundary of the macrotile, then the numbers of the crowns added to the blue index on sides $a$ and $b$ must coincide.
If $a$ is secondary port,
then the numbers of two allowed crowns are added to the blue index of $a$,
first for the left end, then for the right end. The extra information in native blue indices 
may increase the size of the set $\cld_s(A)$,
since there can be several crowns satisfying these restrictions. In particular, even for types $s$ without 
secondary ports,
it might happen that  $|\cld_s(A)|>1$.

To distinguish decorated tiles in this section from the decorated tiles
from the previous section, we will call them \emph{enriched}. The substitution on enriched tiles defined above
will be denoted by $\sigma''$. 
The definition of a legal enriched tile is not changed, but this time we use $\sigma''$ in place of $\sigma'$ 
in that definition. 

We claim that all the lemmas proved so far, and Proposition~\ref{th2}, remain true for substitution $\sigma''$ on 
enriched tiles. Let us check this.
\begin{itemize}
\item Observation~\ref{rem2}.
  For this observation, adding information to the native blue indices  does not matter.
  It only matters how the borrowed indices 
  are defined,
  and we have not changed that.
\item Observations~\ref{o-diff1} and~\ref{o-diff2} in the proof of Dichotomy Lemma.
  They  assert that some blue indices are different.
  Adding information to indices cannot make different indices coincide.
\item Lemma~\ref{l10} (Dichotomy Lemma).
  In cases where the second alternative was true (the tiles are not similar along some side),
  the dissimilarity is preserved when adding information to the indices.
  In cases where the first alternative was true (the tiles are similar),
  the proof of similarity did not use the definition of native blue indices.
\item Lemma~\ref{l9} (Composition Lemma).
  Item (a) does not depend at all on how the blue indices are defined.
  Item (b) could be spoiled by changing blue indices.
  But since Observation~\ref{rem2} is preserved, tiles of all types $s$ except the central one still belong to $\cld_s(D)$.
  And for the central type, the assertion will remain true, since all indices on its sides have not been changed,
  as the central type has no non-inner sides.
  In item (c), it is only important how the borrowed indices on net sides are defined,
  and we did not change that. 
  In the proof of (d), we did not refer to the definition of blue indices,
  so this item remains valid.
\item Proposition~\ref{th2}. In the proof of the proposition, we did not refer to the way in which blue indices are defined at all,
  but only to the lemmas. Since the lemmas are preserved, the proof remains valid.
\end{itemize}

Thus we get the following analogue of Proposition~\ref{th2}
\begin{prop} \label{th2a}
  Every proper tiling $\mathcal T$ with legal enriched tiles has a unique proper $\sigma''$-composition $\mathcal T'$  consisting of legal tiles.
  The tiling  $\mathcal T'$ has the following feature:
all tilings $\tau^{1} \pi(\mathcal T'), \dots, \tau^{m-1} \pi(\mathcal T')$
are side-to-side.\footnote{Actually, this feature will not be used in the sequel. Thus we can
simplify red indices.}
\end{prop}

We begin the proof of Theorem~\ref{th0}(b), as before,
by proving the implication in the easy direction.
\begin{prop}\label{th1a}
  The tiles of any \emph{$\tau$-substitution} tiling $\mathcal T$  can be decorated so as to obtain a proper tiling with
  legal
  enriched tiles.
\end{prop}
\begin{proof}
  This is proved along the same lines as before: by Lemma~\ref{th11} 
  there is a sequence of $\tau$-substitution tilings 
  $$\mathcal T_0=\mathcal T,\mathcal T_1,\mathcal T_{2},\dots$$ witnessing
  that $\mathcal T$ is $\tau$-hierarchical.
  Then the sequence   $$\mathcal T_0,\mathcal T_m,\mathcal T_{2m},\dots$$ witnesses
  that $\mathcal T$ is hierarchical w.r.t. $\sigma=\tau^m$.
  We then decorate tilings $\mathcal T_0,\mathcal T_m,\mathcal T_{2m},\dots$ as before. When doing that, 
  we need to determine the numbers of crowns in native blue indices.
  We define them 
  to be the  numbers of the actual crowns 
in $\mathcal T_{im}$.
  This choice satisfies all the restrictions we imposed on the blue indices of tiles $B\in\cld_s(A)$.
  Indeed, since $\mathcal T_{im}$ is a substitution tiling, those crowns are allowed.
Other restrictions hold by obvious reasons.
\end{proof}

To prove Theorem~\ref{th0}(b), it remains to prove the converse:

\begin{prop}\label{th8}
  If $\mathcal T$ is a proper tiling with legal enriched tiles,
  then  $\pi(\mathcal T)$ is a $\tau$-substitution tiling.
\end{prop}
\begin{proof}
 We first show that all  crowns in the projection of any proper tiling are allowed.
Indeed,  let a vertex $V$ of a proper tiling $\mathcal T$ 
 be given.
 By Proposition~\ref{th2a} it has a  $\sigma''$-composition $\mathcal T' $. In particular, 
$\mathcal T$ 
can be partitioned into $\sigma''$-macrotiles.
  If $V$ is an interior vertex of a macrotile, then the crown in $V$ is contained in a macrotile and therefore is allowed.
  Otherwise, $V$ lies on the boundary of a macrotile. Let us first assume that all sides incident to $V$ do not belong to the net.
  On all such sides, the blue index contains the number of some allowed crown, and this number is the same for all sides.
  Moreover, for each tile with vertex $V$, this number is consistent with the form of this tile.
  Therefore, the crown in this vertex can only be the one whose number is specified in the
  blue indices, and therefore it is allowed.
  
Now  assume  that some net side is  incident to $V$. This side is then a main port. 
Since main ports do not share vertices, there is  only one such side. 
It remains to note that in the previous argument we could admit one non-net side,
  since the crown defined by all other sides is also consistent with the tiles that share this side.

Now we  show that, moreover, the projection of every proper tiling $\mathcal T$ is a $\tau$-substitution tiling.
 By Proposition~\ref{th2a} there exists a sequence $\mathcal T_0=\mathcal T, \mathcal T_m, \mathcal T_{2m},\dots$ of proper tilings   
 in which each tiling is a $\sigma''$-composition of the previous one.
 Then in the sequence $\pi(\mathcal T_0)=\pi(\mathcal T), \pi(\mathcal T_m), \pi(\mathcal T_{2m}),\dots$  each tiling is a composition of the previous 
 one w.r.t. $\sigma=\tau^m$.
As we have shown, all crowns in all these tilings are allowed.
 By Lemma~\ref{l22} the tiling $\pi(\mathcal T_0)$ is a substitution tiling.
Proposition~\ref{th8} is proved. 
\end{proof}
 
Theorems~\ref{th0}(b) and~\ref{th22}(b) are proved.




\section{The Assumptions on Substitution  for the Main Theorem}

The  assumptions for a substitution $\tau$ are obtained by stipulating two 
conditions under which the above technique works.
Now we do not assume that all prototiles are squares.
Instead we   assume that all $\tau$-supertiles are side-to-side tilings:
\renewcommand{\labelenumi}{R\theenumi:}
\begin{enumerate}\addtocounter{enumi}{-1}
\item\label{r0}
All $\tau$-supertiles are side-to-side tilings and hence
every $\tau$-substitution tiling
is side-to-side. 
\end{enumerate}
Under this condition, every
$\tau$-substitution tiling is $\tau$-hierarchical (Lemma~\ref{th11}).  
This is our first assumption.
 
The second  assumption is much more complex: 
we assume that for some $m$ in each $\tau^m$-macrotile one can choose a central type and a net,
similar to central types and nets used in the above arguments.
But this time net paths can bend. In order to make 
main ports  adjacent, we will stipulate certain
requirement called R\ref{r3}. Besides that
the naming of sides can be more complicated, we will stipulate that it is possible to choose
names for sides so that certain conditions are fulfilled.
 
More specifically, we will assume that there exists a ``markup'' of some power of the given substitution $\tau$
meeting certain requirements.
To mark up a substitution we have:
\begin{itemize}
\item 
To assign to each side of each prototile a name so that
different sides of the same prototile have different names.
We  call sides of different tiles with the same name
\emph{eponymous}. When we say \emph{$z$th side of $A$}, we mean  the side of $A$ with the name $z$. The macroside $\sigma(z$th side of  $\alpha_i$)
is called \emph{$z$th macroside of $\mathcal M_i$}.
\item 
To choose a central  type in each macrotile.
\item 
To choose  net paths in each  macrotile.
Each net path in $\mathcal{M}_i$ is a sequence of sides in $\mathcal{M}_i$
such that consecutive sides in this sequence belong to the same tile.
The number of net paths  in $\mathcal{M}_i$ is equal 
to the number of sides of $\alpha_i$, the path $P_z$ begins  with an  outer side 
on 
$z$th macroside of $\mathcal M_i$ and
  ends with a side of the central tile, which must have the same name $z$,
  see Fig.~\ref{namer}. Outer net sides are called \emph{main ports} and other outer sides are called \emph{secondary ports}.
\end{itemize}

\begin{figure}
\begin{center}
\raisebox{1.5cm}{Good:}   \includegraphics[scale= 1]{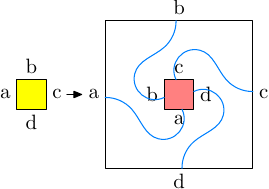}
\hskip 1cm
\raisebox{1.5cm}{ Bad:}  \includegraphics[scale= 1]{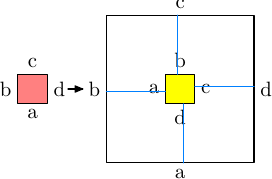}
  \end{center}\caption{The mark up of the second macrotile  
  is bad, as some path connects a side of the central tile with a non-eponymous macroside}\label{namer}
  \end{figure}

The markup must satisfy the following conditions. These conditions
ensure that all Observations, Lemmas and Propositions can be proven in the 
same way as before. After each condition we point to the proof where we need it.
\renewcommand{\labelenumi}{R\theenumi:}
\begin{enumerate}
\item
All sides of the central type are inner sides.
This property was used in the proof of item (b)
in the Composition Lemma for $\sigma''$.   \label{r1} 

\item Net paths cannot share sides. In particular, the central tile in $\mathcal M_i$
must have at least as many sides as the prototile $\alpha_i$. Without this requirement
the definition of blue-green indices on net sides would be incorrect.\label{r2}

\item There must be at least one secondary port on each macroside.
This condition is needed to match red indices of adjacent composed tiles.

\item 
Any  border side does not belong to the net. Otherwise the definition of green indices on border
sides would be incorrect.\label{r7}
Besides, this property is used in the proof of Proposition~\ref{th8}.

\item Different main ports cannot share  a vertex of a tile.
This property is used in the proof of Proposition~\ref{th8}.\label{r5}

\item
In any non-overlapping connection of two macrotiles (without rotation or reflection)
macroside-to-macroside and side-to-side, 
each main port must be adjacent to a main port, and therefore
each secondary port to a secondary port.
This property is used in the proofs of both Propositions~\ref{th1} and~\ref{th2}.
 \label{r3} 

\item Let $A,B$ be prototiles and $a,b$ names of sides.
Call a quadruple $\pair{A,a,B,b}$ \emph{admissible}
if there is a shift $A'$ of $A$ that does not overlap with $B$ and such that 
$a$th side of $A'$ coincides with $b$th side of $B$.

We require that \emph{there is no admissible quadruple of the form 
$\pair{A,z,B,z}$.}
\ie{See the picture:
\begin{center}
 \raisebox{.75cm}{Good:}   \includegraphics[scale= 1]{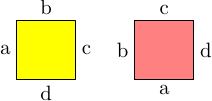}\hskip 1cm  \raisebox{.75cm}{Bad:}   \includegraphics[scale= 1]{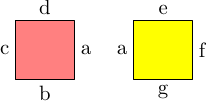}
  \end{center}
 This requirement is needed for Observation~\ref{o-diff1}.\footnote{Usually this requirement is satisfied because
the angle between inward-directed  normals to eponymous sides of different tiles 
is different from $180^\circ$.
(In an admissible connection of sides $a$ and $b$, the normals to the sides
look in opposite directions, so the angle between them is $180^\circ$.)}
}
\label{r4}

\item ``$2/3$ Requirement''.
In any non-central type
less than two-thirds of its sides
belong to the  net. In other words, if a non-central type has $n$ sides, then it is traversed by less than $n/3$ net paths.
For example, if the number of
sides of a tile is 3, then no net path traverses the tile,
if 4--6, then at most one path traverses the tile, if
7--9, then at most two paths, and so on.  See Fig.~\ref{fr6}. \label{r8}
\begin{figure}
 \begin{center}
  \raisebox{2cm}{Good:}  \includegraphics[scale=.5]{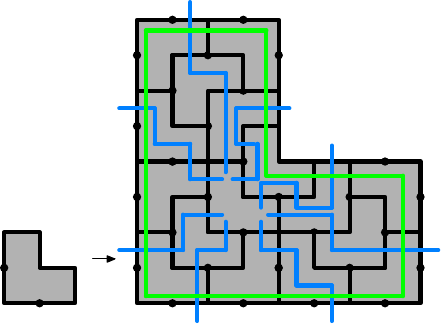}
  \raisebox{2cm}{Bad:}  \includegraphics[scale=.5]{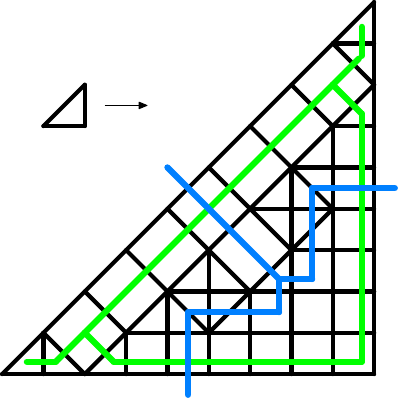}
  \end{center}
  \caption{On the left is a decomposition of a non-convex octagon 
  and its markup satisfying $2/3$ Requirement.
  On the right there is a decomposition of a triangle and its markup that do not satisfy this requirement.
  }\label{fr6}
\end{figure}
Without this condition Dichotomy Lemma might not hold.
\end{enumerate}

\begin{remark}
Without loss of generality we may assume that 
the name of each side of a prototile is 
equal to the equivalence class of this side for the following equivalence relation $R$: 
the relation $R$ is  the transitive closure of the graph of the function 
$$
a\mapsto\text{the side of the central tile
connected by a net path with }\sigma a.
$$
If this naming does not satisfy the requirements, then no naming does.
Therefore we do not specify names of sides in the examples.
\end{remark}

\begin{remark}
Requirements R1--R7 are similar 
 to the requirements for substitution from the paper
of Fernique --- Ollinger. 
However, as we will see below, these conditions are not enough.
\end{remark}

We provide two examples of markups of non-rectangular macrotiles  that meet
all the requirements.

\begin{ex}
The markup of the third substitution is shown in Fig.~\ref{f7}.
\begin{figure}
\begin{center}
\includegraphics[scale=.34]{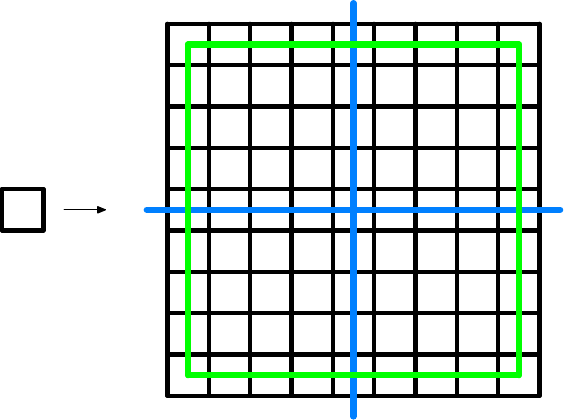}
\includegraphics[scale=.34]{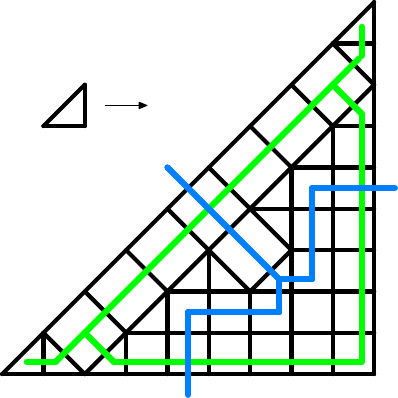}
\includegraphics[scale= .34]{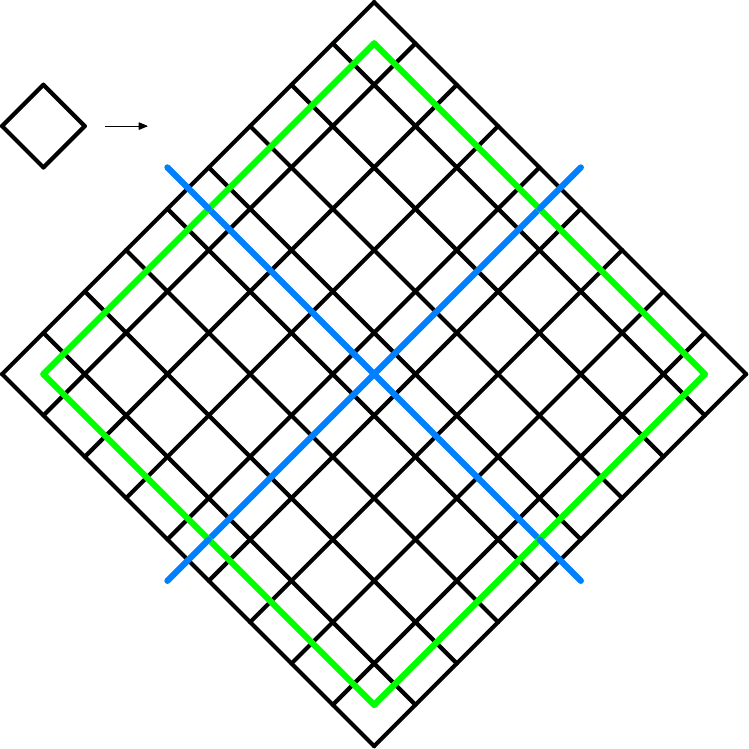}
\end{center}\caption{The markup for the macrotiles
of the third substitution.%
}\label{f7}
\end{figure}
For this markup, all sides
have different names. 
This example is interesting because some prototiles are triangular.
The net paths are forced to bypass triangular tiles to satisfy $2/3$ Requirement.
\end{ex}

\begin{ex}
Fig.~\ref{f55} shows a markup of another substitution,
\begin{figure}
\begin{center}
\includegraphics[scale=.4]{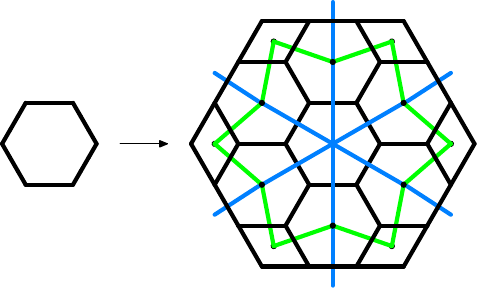}
\hskip .5cm
\includegraphics[scale=.4]{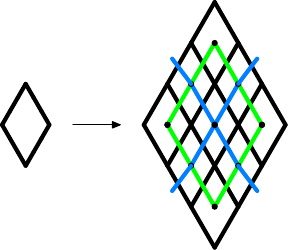}
\hskip .5cm
\includegraphics[scale= .4]{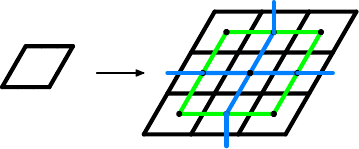}
\hskip.5cm
\includegraphics[scale=.4]{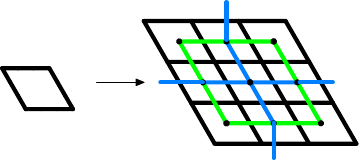}
\end{center}\caption{The fourth substitution and its markup.
}
\label{f55}
\end{figure}
satisfying all the requirements. 
\end{ex}

Yet another good markup is shown on the left in Fig.~\ref{fr6}. 


\section{Main Theorem}

\begin{theorem}\label{th-main}
Assume that a substitution $\tau$ 
satisfies R0 and some its power admits a markup satisfying conditions R1--R8. Then
(a) the 
family of \emph{$\tau$-hierarchical} tilings is sofic and 
(b) the 
family of \emph{$\tau$-substitution} tilings is sofic.
\end{theorem}

In the rest of this section we prove this theorem.

Let $\tau^l$ be the power of $\tau$ that  admits a markup satisfying conditions R1--R8.
We first  show that some power  $\tau^{nl}=\tau^m$ of $\tau^l$
has a markup satisfying 
Cyclicity Requirement (page~\pageref{r9}), Outer Sides Requirement  (page~\pageref{outersides}) and 
all requirements R1--R8 except for~R\ref{r3}. Requirement~R\ref{r3}
will be met in the following reduced form:
\begin{enumerate}
\item[R\ref{r3}':] 
If  tiles $A$ and $B$
share a side $a$ and \emph{for all $i=1,\dots, m$ the tiling $\tau^i \{A,B\}$ 
is side-to-side}, then the main port in $\tau^m A$  
on the superside  $\tau^m a$ matches  a main port of $\tau^m B$.\footnote{Actually, 
R\ref{r3}'  is a property of the pair $\pair{\tau,m}$ rather than 
a property of $\sigma=\tau^m$.}  
\end{enumerate}


\begin{lemma}\label{l66}
Some power $\sigma=\tau^m=\tau^{nl}$ of $\tau$ meets 
all requirements R1--R5,  R7, R8, 
Cyclicity Requirement, Outer Sides  Requirement and~R\ref{r3}'. 
\end{lemma}
\begin{proof}
Let $c$ denote the function acting 
central types for substitution $\tau^l$.
That is, $c(t)$ is the 
central type in the macrotile $\tau^l t$.
As we know from the proof of Lemma~\ref{l6},
for some $n>0$ we have $c^{2n}(t)=c^n(t) $ for all $t$. 
In other words, $c^n(t)$ is a fixed point of the function $c^n$ for all $t$.

Consider the $n$-th power of the substitution $\tau^l$ for any $n>0$ with  $c^{2n}=c^n$.
Let us mark up $\tau^{jl}$-macrotiles  for $j=1,2,\dots,n$
recursively:\\ 
(1) The markup for  $\tau^l$ is the given markup.\\
(2) We keep the same names of sides for
 $\tau^{(j+1)l}$ as they are for $\tau^{jl}$. 
Recall that the $i$th $\tau^{(j+1)l}$-macrotile   $\mathcal M'_i$ 
  is obtained from $i$th  $\tau^{jl}$-macrotile $\mathcal M_i$ 
by replacing each tile by the corresponding macrotile for $\tau^l$, see Fig.~\ref{f4-9}.
\begin{figure}
  \begin{center}
    \includegraphics[scale=1.1]{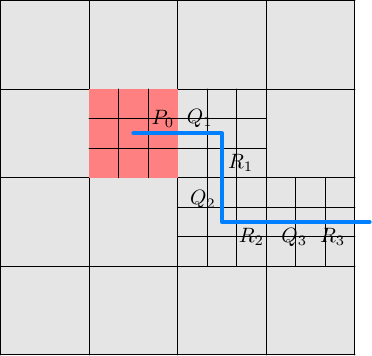}
  \end{center}
  \caption{The markup of substitution  $\tau^{(j+1)l}$.
The figure shows a part of a macrotile for $\tau^{(j+1)l}$.
Three by three squares represent macrotiles for the original substitution $\tau^{l}$.
The central macrotile is pink.}\label{f4-9}
\end{figure} 
As the central tile of $\mathcal M'_i$  choose
the central tile in the macrotile $\mathcal M$ 
which replaces the central tile of $\mathcal M_i$.
Then choose the net paths in $\mathcal M'_i$  as follows.
Let $a_0,a_1,\dots,a_k$ denote the sides of a net path in $\mathcal M_i$, where 
 $a_0$ is a side of the central tile and $a_k$ is a main port.
 Let $A_0,A_1,\dots,A_k$ denote the tiles from $\mathcal M_i$
 which these sides belong to.
Let $P_0$ denote the net
path in $\mathcal M=\sigma^l A_0$
 that ends on the macroside $\tau^l a_0$ at a main port $b$. 
Let $Q_1$ denote the net path in the macrotile $\tau^l A_1$
that connects the main port $b$ with its central tile.
Let $R_1$ denote the net path in  $\tau^l A_1$
from its central tile
to its macroside shared with $\tau^l A_2$.
And so on. Our path is the concatenation 
of paths $P_0,P_1,Q_1,\dots,P_k,Q_k$.
Recall that we assume that the markup for $\tau^l$ 
satisfies the Requirement R\ref{r3} and that all supertiles are 
side-to-side tilings. This implies that all these paths exist
and their concatenation is a path.

Let us verify all the requirements for this markup of $\tau^{nl}$.
\begin{itemize}
\item Cyclicity Requirement holds by the choice of $n$.
\item R1--R5  for $\tau^{nl}$ follows from that for $\tau^l$.
\item R\ref{r3}' holds by construction.

\item R\ref{r4} holds, as  
 sides of prototiles keep their names.
\item R8 ($2/3$  Requirement). 
Let us show by induction that all
substitutions $\tau^{jl}$ for $j\le n$ meet this requirement. 
Let $B$ be a non-central tile from a macrotile $\tau^{(j+1)l} C$. 
Denote by $A$ the tile in $\tau^{jl} C$ such that $B\in \tau^l A$.
Then $B$ is a non-central tile in $\tau^l A$ or $A$ is a non-central tile in 
$\tau^{jl} C$. 

In the first case, in the markup of $\tau^{(j+1)l} C$,  only those 
net paths can include $B$ whose subpaths lying in  $\tau^l A$ include $B$.
Those subpaths are pair wise different by Requirement R\ref{r4}, since they connect the central tile of $\tau^l A$ 
to different macrosides of $\tau^l A$.
The number of such subpaths 
is at most one third of $B$'s sides by $2/3$  Requirement for $\tau^l$.

Otherwise $B$  is 
the central tile in $\tau^l A$ and $A$ is a non-central tile in 
$\tau^{jl} C$. 
The number of net sides of $B$ is twice the 
number of net paths from $\tau^{jl} C$ which $A$ belongs to. The latter is at most one third of the number 
of sides of $A$, as $\tau^{jl}$ meets $2/3$  Requirement by Induction Hypothesis.
By the Requirement R\ref{r2} for $\tau^l$, the number of sides of $B$ is at least that of $A$ 
hence $2/3$  Requirement holds for $B$.

\item Outer Sides Requirement. 
Note that $n$ can be chosen arbitrary large.
Choose $n$ 
so large that the substitution $\sigma=\tau^{nl}$
has the following property: if a type $t$ 
from a macrotile   has an outer side $a$ on a macroside $U$,
then all its outer sides lie on $U$  or on the macroside $V$ 
that shares a vertex with $U$  and the tile itself. 
  \begin{center}
    \includegraphics[scale=1]{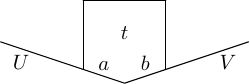}
  \end{center}
This property implies Outer Sides Requirement. Indeed, if $b$ is another outer side of $t$, then
$b$ lies on $U$ or on $V$. 
In the first case $t$
is either to the left of both sides $a,b$, or to the right. The same 
 holds if $b$ lies on $V$ and the angle between $U$ and $V$ is $180^\circ$.
Otherwise $a$ and $b$ are not parallel.
\qed
\end{itemize}
\renewcommand{\qed}{}
\end{proof}
  
We fix any $m$ satisfying this lemma 
and consider the substitution
$\sigma=\tau^m$. 
The second step is to define the set of decorated tiles. 
For both statements (a) and (b) the substitutions  $\sigma'$ and  $\sigma''$
are defined in the same way as before. 
\begin{remark}
The substitution $\sigma'$ on decorated tiles 
is similar to the Fernique --- Ollinger substitution from~\cite{fo}.
The main difference is in the blue indices, which in our construction
carry more information. Due to this, Observation~\ref{o-diff2} in the proof of the Dichotomy Lemma
below holds. That observation
may not hold for the Fernique --- Ollinger construction.
\end{remark}

Then we define legal tiles and enriched legal tiles
exactly as before.
All properties of legal tiles, including Dichotomy Lemma, remain valid
with a little modification of the statement of Observation~\ref{rem2}.
Now it reads:

\textbf{Observation~\ref{rem2}.}
 \emph{Let a type $s$ from a $\sigma$-macrotile $\mathcal{M}_i$ be fixed. 
Then the set $\cld_s(A)$ 
depends only on the tuple consisting  of blue-green indices
on $z$th sides of  $A$ such that  the net path $P_z$ passes through tile $s$
and on the type of 
$A$ and its red indices provided 
$s$ is a border type.}

All those properties are proved in a similar way except for Dichotomy Lemma, whose proof
is more complicated in general case.
 Requirements R7 and R8 are needed only in that proof. Let us remind its statement:

\smallskip
\textbf{Dichotomy Lemma.}
\emph{Let a legal tile $D$ and a central cyclic or non-central type $t$
of the same form as $D$ be given.
Then 
either $D$ is similar to some legal tile of type $t$,
or there is a name $z$ such that
$D$ is not similar to \emph{any} legal tile of type $t$ on $z$th side.}

\begin{proof}[Proof of the Dichotomy Lemma for $\tau^m$]
First note that Observations~\ref{o-diff1} and~\ref{o-diff2} 
remain valid, as their proofs are valid also
for non-rectangular tiles under R7 and Outer Sides Requirement.

Assume first that $t$ is a central cyclic type. 
In this case the first alternative holds, which is proved in  the same way, as before.

Otherwise $t$ is a non-central type.
Let 
$$
\dots\to D_2\to D_1\to D_0=D
$$ 
be a sequence of tiles that prove legality of $D$. 
Again we distinguish  two cases.  

\emph{Case 1:} for some $k=0,1,\dots$ the tile $D_k$ has a non-central type, say type $s$.
Consider the minimal  such $k$. 
If $s=t$, then the first alternative holds, since the blue-green
index of $D_k$
passes unchanged to $D$.
We claim  that otherwise the second option holds.

First note that for any side $a$ of the tile $D_k$, the tile $t$
has a side eponymous to $a$.
Indeed,
in the sequence $D_{k},D_{k-1},\dots,D_0=D$ each tile 
 is the central
child of the previous one and hence inherits from it names of sides. The tile $t$ also has such a side,
since by assumption  $D$ and $t$ are of the same form.   

Now we consider three cases
in which we can prove the required statement,
and then we establish that one of the cases always holds.

\begin{enumerate}
\item[(a)] There is a  \emph{non-net} side in type $t$ such that 
 $D_k$ has no eponymous side.
  Let $z$ denote the name of that side. Since $D_0=D$ and $t$ have the same form,
  the tile $D_0$ has a side named $z$.
  Consider the smallest $i\le k$ for which  $D_i$
  has no $z$th side.
   \begin{center}
    \includegraphics[scale=1]{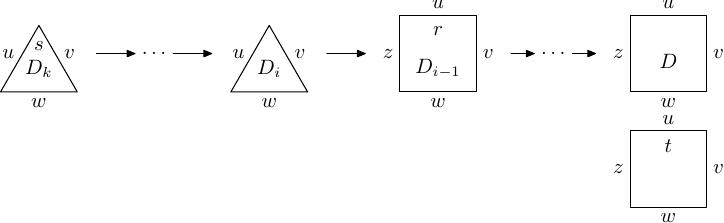}
  \end{center}
  Then  the tile $D_{i-1}$  has $z$th side. As $i\le k$ the tile $D_{i-1}$  is of some central type $r$.
  All net sides of $D_{i-1}$ 
borrow their names from $D_i$ hence $z$th side of $D_{i-1}$ does not belong to the net. 
 Being central, the type $r$  is different from $t$.
By Observation~\ref{o-diff1} 
the blue index on the $z$th side of $D_{i-1}$ (and hence on $z$th side of $D$)
is different from that of any legal type-$t$ tile.

\item[(b)] There is a \emph{non-net} side in type $t$
  which has eponymous  \emph{non-net} side in $D_k$. 
Then again by Observation~\ref{o-diff1}
the blue index on that  side  of $D_{k}$ (and hence of $D$)
is different from from that of any legal type-$t$ tile (recall that $s\ne t$).

\item[(c)] There are two different sides with names $a,b$ in the tile $D_k$ 
that belong to the same net path and
both $a$th and $b$th sides of the tile $t$ are \emph{non-net} sides.
 \begin{center}
    \includegraphics[scale=1]{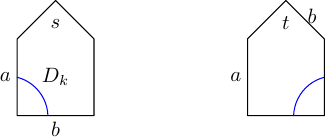}
  \end{center}
 Blue indices of $a$th and $b$th sides of $D_k$ coincide, denote them by $y$.
The same is true for tile $D$, since $D$ inherits these indices from $D_k$.
By Observation~\ref{o-diff2} there is $z\in\{a,b\}$ such that
$y$ is different from the blue index on $z$th side of any legal type-$t$ tile. 
\end{enumerate}

Why does one of  cases (a), (b) or (c) always happen? Let the number
of sides in tiles $D_k$ and $t$ be $n$ and $l$, respectively. As explained above, $n\le l$.
Suppose that (c) does not hold.
We claim that (a) or (b) then holds. To prove this, it suffices to establish that
the number of names $z$  for which $z$th side 
is a net side  in $D_k$ or in $t$ is less than $l$.

This number is equal to the sum of two numbers:
the number $T$ of names $z$ for which the $z$th side  belongs to
the net in tile $t$, plus the number $U$
of names $z$ for which the $z$th side  
is in the net in $D_k$ but not  in $t$.
The number $T$ is less than $2l/3$ by $2/3$ Requirement.
And the  number $U$ is less than $n/3$. Indeed, the net sides
in tile $D_k$ are grouped  into pairs, each pair of sides
belongs to one path of the net. By $2/3$ Requirement the number of pairs
is less than $n/3$. Moreover, in each pair at most one side
can contribute to $U$, since we assume the negation of (c).
In total, we get less than $2l/3+n/3\le l$.

\emph{Case 2:} for all $k=0,1,\dots$ the tile  $D_k$ is of central type. We claim that then the second alternative holds.

To prove the claim, we use a simple corollary of $2/3$ Requirement:
\emph{any non-central type has at least one non-net side.}
Therefore, the type $t$ has a side on which the blue index is
native. Choose any such side and denote by $z$ its name.
We claim that the blue index on $z$th side of $D$ is different from that of any legal type-$t$ tile.

To prove the claim, choose any legal type-$t$ tile $B$.
Since the blue index on $z$th side
of  $B$ is native, it is different from $0$. If
 $D$ has zero blue index on $z$th side, then we are done. Otherwise, it
originates in some tile $D_k$. In that tile, $z$th side does not belong
to the net. Since $D_k$ has a central type, its type is different from $t$. By Observation~\ref{o-diff1} 
the blue index of $B$ on $z$th side
is different from that of $D_k$ and hence of $D$.

The lemma is proved.
\end{proof}

\begin{remark}
It follows from the proof of the Dichotomy Lemma that $2/3$ Requirement
can  replaced by the following weaker condition:
\renewcommand{\labelenumi}{R\theenumi:}
\begin{enumerate}\addtocounter{enumi}{4}
\item[R8':] 
For any ordered pair $\pair{s,t}$ of  different non-central types
 at least one of the following three conditions holds:
\begin{itemize}
\item There is a side in $s$ with no eponymous side in $t$.
\item There is a non-net side in $t$ with no 
  eponymous \emph{net} side in $s$.
\item There are two sides in $s$ that belong to the same net path for which
  both eponymous sides in $t$ do not belong to the net.

\end{itemize}
\end{enumerate}
We first derived condition R8' from $2/3$ Requirement,
and then used it in the analysis of Case 1. The advantage of the stronger $2/3$ Requirement over R8'
is that it can be verified much faster.
\end{remark}

The rest of the proof of Theorem~\ref{th-main}
is as before. The only thing worth to notice is that in the very end 
of the proof of (b) we use R4  and R5.

\section{On the Applicability  of the Main Theorem}

It may seem that the requirements R1--R8 greatly limit the applicability of Theorem~\ref{th-main}.
But in fact this is not the case, as we will try to convince the reader.
The following two obstacles usually prevent the desired markup of macrotiles.
(1) The lack of space in macrotiles for placing net paths and (2) 
a large number of tiles $A$ such that there is  a tile $B$ which shares one third or more of its 
sides  with $A$, for instance, a large number of triangular tiles; this makes $2/3$ Requirement hard to satisfy.
Let us consider in turn how to overcome these obstacles.

\subsection{Increasing Space in Macrotiles}
The first obstacle is the lack of space in macrotiles. 
Usually
it can be overcome by considering a sufficiently large power of the original substitution.
%
We will demonstrate this using the Chair Substitution~\cite{ra,gs1} (see Fig.~\ref{f222}).
For this substitution, it is impossible to mark the macrotiles satisfying all the requirements.
\begin{figure}
  \begin{center}
    \includegraphics[scale=.5]{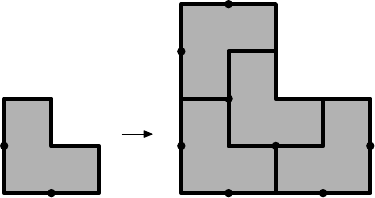}
  \end{center}
  \caption{Chair Substitution. The prototiles are four non-convex octagons.
    The figure shows one of them, and the other three are obtained by rotating  by 90, 180 and 270 degrees and
    are decomposed similarly.}
  \label{f222}
\end{figure}
However, for the square of this substitution this is already possible.
For one of the four macrotiles, the markup is shown in Fig.~\ref{fr6} on the left.
For the other three, one can proceed similarly, keeping in mind
that it is necessary to select as main port matching sides, for example,
the second side from the left
on the respective macroside (the second from the bottom --- for vertical macrosides).
This will ensure R6.
Requirement R\ref{r4} is satisfied, since the names of all sides are different.
$2/3$ Requirement is satisfied, since all tiles have eight sides and each non-central
type has at most $4<(2/3)\cdot 8$ net sides.

\subsection{Joining  tiles}

The second obstacle is the $2/3$ Requirement.
For example, if all prototiles are triangular, then it is impossible to fulfill.
This obstacle can be bypassed by joining  some tiles into larger ones.
We will demonstrate this with two examples.
As a first example, consider the substitution in Fig.~\ref{fer9}.
\begin{figure}
  \begin{center}
    \includegraphics[scale=.7]{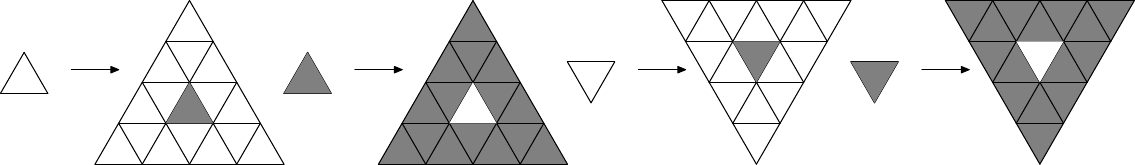}
  \end{center}
  \caption{Substitution for triangular tiles.}\label{fer9}
\end{figure}
In any side-to-side tiling of the plane, each tile facing up is adjacent to the right by a tile facing down.
We join them into a parallelogram.
The resulting tiling will be defined by a substitution with four rules,
one of which is shown in Fig.~\ref{fer10}, and the other three are similar.
\begin{figure}
  \begin{center}
    \includegraphics[scale=.7]{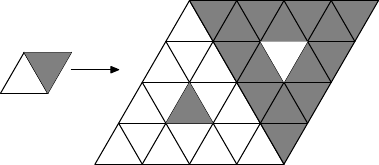}
  \end{center}
  \caption{One of the four rules for the substitution obtained by joining two triangular tiles into a parallelogram.}
  \label{fer10}
\end{figure}
The new substitution can be easily labeled so that all conditions R1--R8 are satisfied.
Therefore, both the family of hierarchical tilings and the family of substitution tilings are sofic for it.
This implies that both families are sofic for the original substitution as well.
Indeed, tilings with the original tiles are obtained by cutting each tile into two parts.
The cutting side must be labeled with some unique color, which will force triangles to assemble into parallelograms.

As a second example, consider Robinson's Stone Inflation \cite{gsh} (Fig.~\ref{prt11}).\footnote{For this substitution,
the families of substitution and hierarchical tilings coincide and can be defined
by well-known simple local rules: red and purple circles in adjacent tiles must continue each other.
We give this example not to derive new local rules for this family,
but only to demonstrate the applicability of the main theorem.}
The macrotiles for this substitution are too small and also consist of triangles.
For both reasons, they cannot be labeled to satisfy R1--R8.
\begin{figure}
  \begin{center}
    \includegraphics[scale=1]{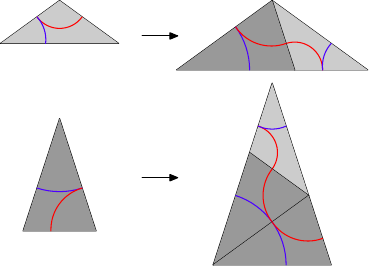}
  \end{center}
  \caption{Robinson's Stone Inflation. Two of the 40 prototiles in this set are shown on the left.
    The remaining 38 are obtained by flipping and rotating by angles that are multiples of 36$^\circ$.
    On the right are their decompositions under the substitution.
    The purple and red circles are drawn to make the tiles asymmetrical
    (one circle on each tile would be enough for this,
    but two cycles  will be more convenient).}\label{prt11}
\end{figure}
Let us consider, however, the sixth power of this substitution.
Two macrotiles for the resulting substitution are shown in Fig.~\ref{prt3}.
\begin{figure}
  \begin{center}
    \includegraphics[scale=.5]{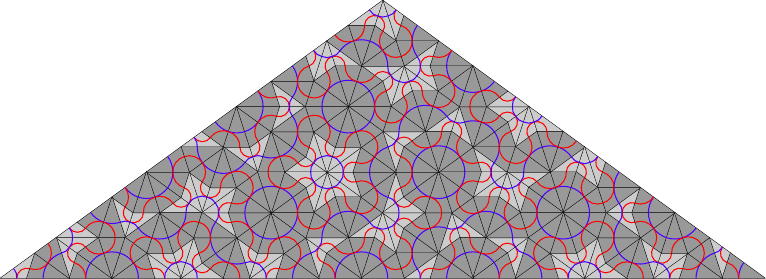}
    \includegraphics[scale=.5]{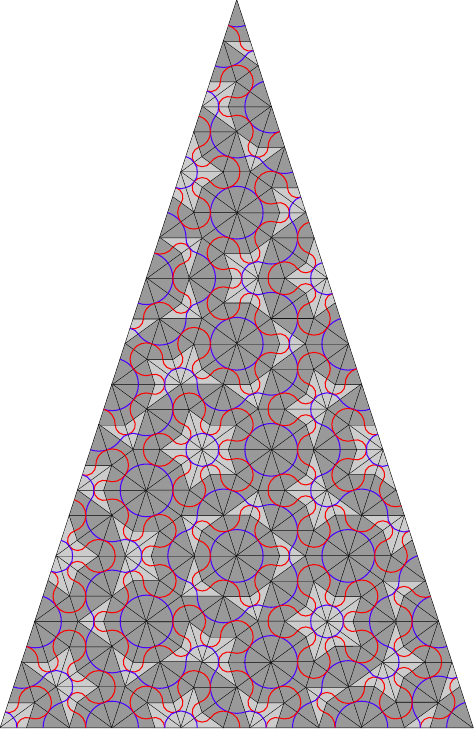}
  \end{center}
  \caption{Macrotiles for the sixth power of Robinson's Stone Inflation.}\label{prt3}
\end{figure}
We group the tiles in these macrotiles into purple and blue rhombuses, as shown in
Fig.~\ref{prt7}.\footnote{
If we tile the plane with them so that the red and purple circles are continuations of each other,
we obtain the family of 
Penrose tilings P2 \cite{gsh}.}
Some tiles will be outside the groups, they are painted white. As a result, we obtain tilings with 80 prototiles: 40 original tiles, 20 blue rhombuses and 20 purple rhombuses
(like triangles, rhombuses have two orientations).
We now extend the substitution to the rhombuses as follows:
first, small rhombuses are cut into two triangles and large ones into four,
then the original substitution is applied to each triangle and then in the resulting macrotile
triangles are again combined into rhombuses.
As a result, we obtain a substitution acting on 80 prototiles.
The resulting macrotiles can now be easily labeled,
satisfying   R1--R8. This labeling is shown in Figs.~\ref{prt7} and~\ref{prt15}.
Therefore, the families of hierarchical and substitution tilings  are sofic for the resulting substitution,
and so are the families of hierarchical and substitution tilings
associated with the original substitution.
\begin{figure}
  \begin{center}
    \includegraphics[scale=.7]{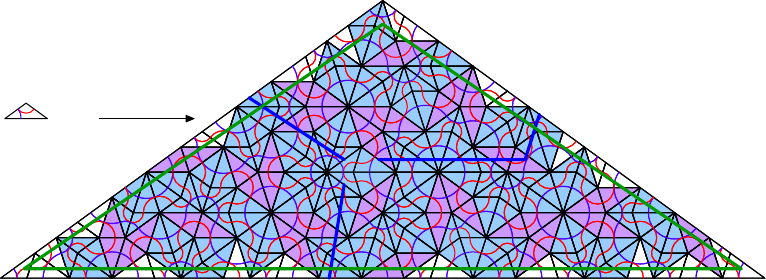}
    \includegraphics[scale=.7]{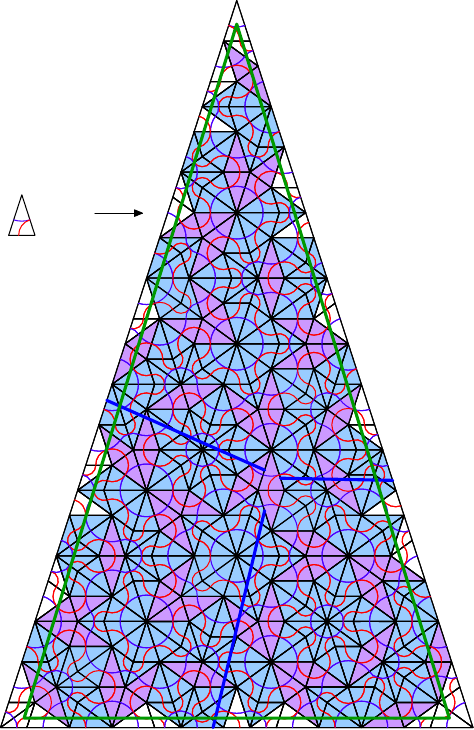}
  \end{center}
  \caption{Labeling of two macrotiles for the sixth power of Robinson's Stone Inflation. }\label{prt7}
\end{figure}
\begin{figure}
  \begin{center}
    \includegraphics[scale=.45]{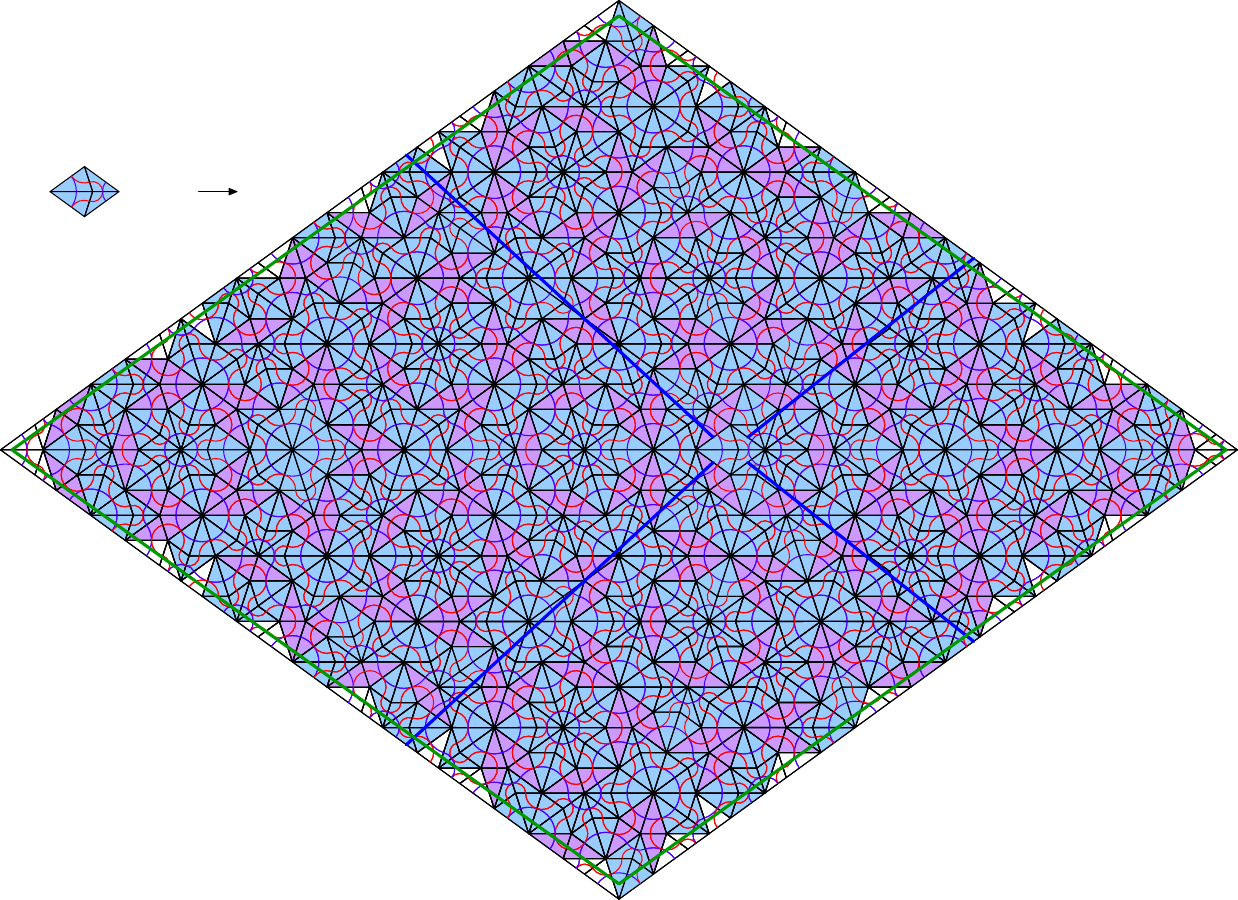}
    \includegraphics[scale=.45]{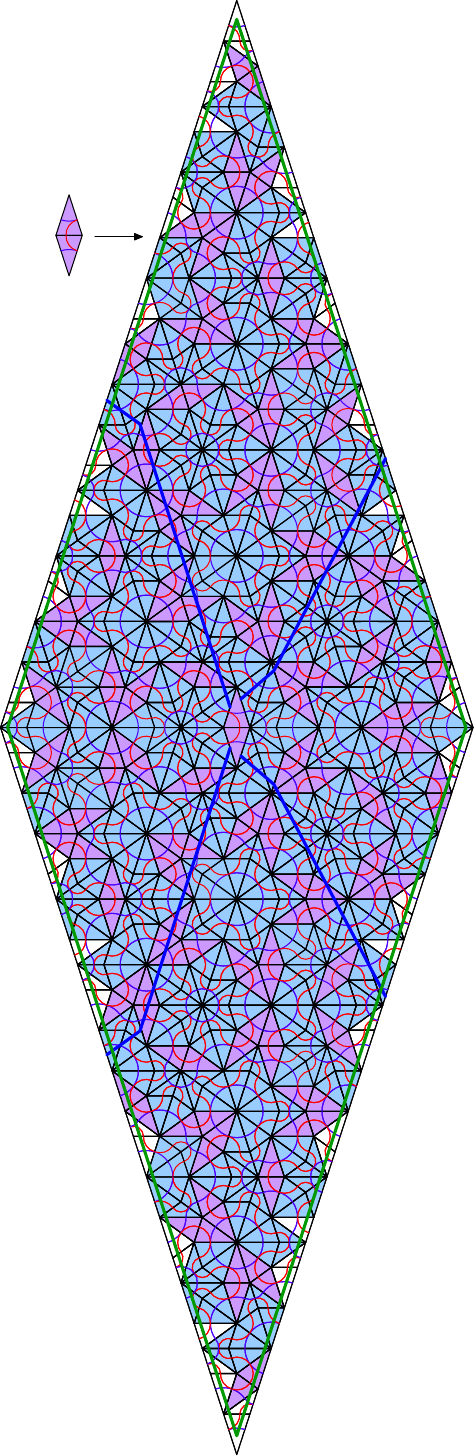}
  \end{center}
  \caption{Labeling of the other two macrotiles for the sixth power of Robinson's Stone Inflation.
    The names of the sides of all macrotiles are different with the following exceptions.
    The central tile in the fourth macrotile (the small rhombus) is oriented in the opposite direction than the parent tile.
    Therefore, the names of the sides of two small rhombuses with different orientations are the same;
    this could be avoided by choosing a different central tile, but then the picture would not be as nice.
    The second exception is that the sides of the triangular tiles have the same names as the sides of the central tiles
    in their decompositions under the substitution.
    Both exceptions do not violate the Names Requirement,
    since the angle between the eponymous sides is different from $180^{\circ}$. }\label{prt15}
\end{figure}

Thus from the main theorem  we obtain the following corollary.
\begin{theorem}
  The families of hierarchical and substitution tilings corresponding to the substitutions in Figs.~\ref{f13}, \ref{f14}, \ref{f55}, \ref{f222}, and~\ref{fer9}   are sofic.
\end{theorem}

\appendix
\section{Appendix}
We discuss here this technique  in more detail. Namely, we explain why some of the assumptions on substitution
cannot be omitted.
In more detail, we prove the following statements:
\begin{itemize}
\item The statement of Dichotomy Lemma can be weakened in the following sense:
if a substitution $\tau$ satisfies R0 and some its power $\sigma$  has a markup satisfying R1--R6 and  the
substitution $\sigma'$ defined above satisfies   the Weak Dichotomy Lemma,
then both families of $\tau$-hierarchical and $\tau$-substitution tilings are sofic.

\item   This technique cannot work without  the Weak Dichotomy Lemma: if the lemma does not hold,
 then item  (d) of the Composition Lemma becomes false: there is an illegal tile 
 whose decomposition consists entirely of legal tiles.
\item For the Weak Dichotomy Lemma to hold, $2/3$ Requirement cannot be omitted.
\item The Weak Dichotomy Lemma cannot be replaced by the following simpler statement:
  \emph{ any two legal tiles of different non-central types have distinct blue contours}.
\item For central acyclic types the item (d)  in the Composition Lemma can
be false.
\end{itemize}

\subsection{Weak Dichotomy Lemma}


\begin{itemize}\item[]
  \textbf{Weak Dichotomy Lemma.}
  \emph{Let a legal tile $D$ and a non-central or central cyclic type $t$
of the same form as $D$ be given. Assume that 
 for every border type $s$  in $\sigma D$
 there exists a legal tile of type $t$ similar to $D$ on $z$th  side for all $z$ such that  the path $P_{z}$
  contains $s$. Then $D$ is similar to some legal tile of type $t$. }
\end{itemize}
Together with Conditions R1--R6 this lemma is sufficient for Theorem~\ref{th-main}.
This is proved in the same way as before, only item (c) of the Composition Lemma should be reformulated as follows:\\
\emph{For any border type $s$ there is a legal tile of type $t$ similar to $D$ on
 $z$th side for all $z$
such that the path $P_{z}$ contains $s$}.\\
This statement is proved in the same way as before.

\subsection{Necessity of the Weak Dichotomy Lemma}
We now show that if the Weak Dichotomy Lemma fails, then item (d) of the Composition Lemma fails.
Indeed, assume that the Weak Dichotomy Lemma fails. That is, there exists a legal tile $D$ and a non-central or central acyclic type $t$
such that\\ 
(1) $t$ and $D$ are of the same form,\\ 
(2) For any border type $s$ there is a legal tile of type $t$ similar to $D$
 on $z$th side for all $z$
such that the path $P_{z}$ contains $s$, and\\ 
(3) $D$ is not similar to any legal tile of type $t$.\\
Then consider the following normal tile $D'$:
its central index is $t$ and it has the same blue-green contour as $D$.

By item (3) this tile $D'$ is illegal. However,
the decomposition of $D'$ consists of legal tiles.
Indeed, by Observation~\ref{rem2},  all non-border tiles in the decomposition of
$D'$ are children of $D$,
and so are legal. Now let $B$ be a tile in the decomposition of $D'$ of a border type $s$.
By (2) there exists a legal tile $C$ of type $t$ that is similar to $D$ and hence to $D'$
on $z$th side for all $z$
such that the path $P_{z}$ contains $s$.
By Observation~\ref{rem2}, $B$ is the child of $C$ and hence is legal.

\subsection{Necessity of $2/3$ Requirement}\label{sa2}

One of the key assumptions required by the Dichotomy Lemma is $2/3$ Requirement or its weak version~R8'.
Here is an example of a substitution $\sigma$ and its labeling for which all requirements except these two are satisfied,
but the Weak Dichotomy Lemma is not (for some $D,t$). As shown in Section A.2, for such $\sigma$
there exists an illegal normal tile $D'$ of type $t$ in whose $\sigma'$-decomposition all tiles are legal.

The substitution $\sigma$ is similar to the substitution in Fig.~\ref{f7}.
The difference is that four more squares in the decomposition of the triangular tile are cut into triangles;
the network paths pass through triangles eight times (see Fig.~\ref{fi7}).
\begin{figure}
  \begin{center}
    \includegraphics[scale=.8]{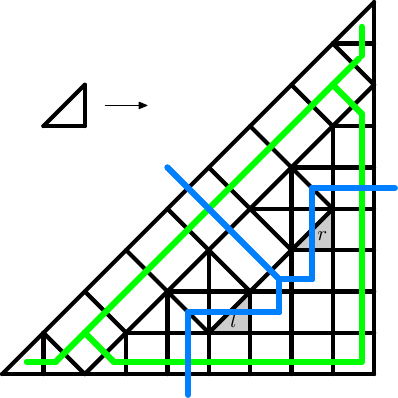}
  \end{center}
  \caption{The markup  of the triangular macrotile that does not satisfy the Weak Dichotomy Lemma,
    but satisfies R1--R7.
    The pair of triangular types $\pair{r,l}$ for which R8'
    is not satisfied is highlighted in gray5.} \label{fi7}
\end{figure}
Namely, condition~R8' is not satisfied for the pair of types $\pair{ r,l}$ highlighted in gray in Fig.~\ref{fi7}. The first condition in requirement R8' is not satisfied, as 
these types are of the same form.
The only non-net side of tile $r$ is the vertical leg, and it belongs to the net in tile $l$.
Therefore, the second condition in requirement R8' is not satisfied.
On the other hand, tile $l$ has only two net sides, of which the hypotenuse belongs to the net in tile $r$ as well,
so the third condition in requirement R8' is not satisfied either.

Let us now show that the Weak Dichotomy Lemma is false for this substitution.
Let $\alpha$ denote the prototile on the left in Fig.~\ref{f35}.
Let $\pair{a,0}$ denote the native blue-green index on horizontal leg of legal tiles of type $l$, and
$\pair{b,0}$  the native blue-green index on vertical  leg of legal tiles of type $r$.
Let us prove that both blue-green indices $\pair{a,0}$, $\pair{b,0}$ occur on both legs of legal triangular tiles of all four forms (but not all types).
When applying the substitution, the blue-green index from the vertical leg of a triangular tile goes to the vertical leg of some
triangular tile rotated by $90^\circ$ (counterclockwise).
Thus, the same blue-green indices occur on the vertical legs of triangular tiles of all four forms.
Similarly, the same is true for horizontal legs.
In addition, when applying the substitution, the blue-green index from the vertical leg goes to the horizontal leg of some
tile and vice versa. Consequently, both blue-green  indices $\pair{a,0}$, $\pair{b,0}$ occur on all legs of legal triangular tiles of all four forms.

Consider the tiles shown in Fig.~\ref{f35}.
\begin{figure}
  \begin{center}
    \includegraphics[scale=1]{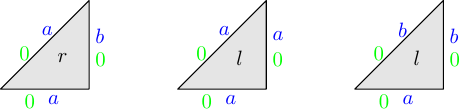}
  \end{center}
  \caption{The legal  tile $D$ of type $r$ (on the left) is similar on each side to one of two legal
  tiles of type $l$.}
  \label{f35}
\end{figure}
All three tiles are legal: the first one is the child of type $r$ of any legal
tile of the form $\alpha$ with blue-green index $\pair{a,0}$ on the vertical leg,
the second one is the child of type $l$ of any legal
tile of the form $\alpha$ with blue-green index $\pair{a,0}$  on the horizontal leg,
and the third one is the child of type $l$ of any legal
tile of the form $\alpha$ with blue-green index $\pair{b,0}$ on the horizontal leg.

Let us take the first tile in Fig.~\ref{f35} as the tile $D$ and $l$ as the type $t$.
Then $D$ is similar on each side to some tile of type $l$.
More precisely, it is similar on the hypotenuse to the second tile,
on the vertical leg to the third tile, and on the horizontal leg to both tiles.
On the other hand, no legal tile of type $l$ can have the same blue-green contour as $D$,
since the blue indices on the hypotenuse and on the vertical leg are different.

This example also shows that the Weak Dichotomy Lemma is not implied  by the following simpler statement:
\emph{any two legal tiles of different non-central types have distinct blue contours}.
Indeed, for this substitution this simpler statement is true.
We will prove this, say, for types $l$ and $r$; the general case is similar to this one. 

Indeed, assume that the blue contours of a
tile of type $l$ and of a tile of type $r$ coincide.
Tiles of type $l$  have identical blue indices on the hypotenuse and on the vertical leg,
and tiles of type $r$ have identical blue indices on the hypotenuse and on the horizontal leg.
Therefore, all three indices are the same for both tiles. However, this is impossible because the first and second tiles
have one native index each and they are different.

\subsection{Item (d) of Composition Lemma 
May Be False for  Central Acyclic Types}\label{sa3}

Consider the substitution in Fig.~\ref{f355}, the central types are located in the middle of the macrotiles.
Let $t$ be the central type of the second macrotile
and $D$ any legal tile in the upper left corner of the second macrotile
(see Fig.~\ref{f355}).
Note that $t$ is an acyclic type.
\begin{figure}
  \begin{center}
    \includegraphics[scale= 1]{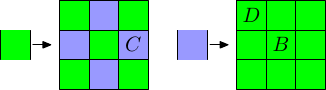}
  \end{center}
  \caption{It can be proved that $B$ cannot have the same blue-green contour as tile $D$.
    However, knowing only $D$ but not knowing the contour of $B$, we cannot find the side on which their blue-green indices differ.}
  \label{f355}
\end{figure}
Let $B$ be a  legal tile of type $t$.
Given the tile $D$ but not knowing the type of the tile from which $B$ inherits its blue-green contour,
we cannot find a side on which tiles $D$ and $B$ have different blue-green index.

We claim that for each side of $D$ there is a legal  type-$t$ tile which is similar to $D$ on that side.
Indeed, each of the four blue-green indices of $D$ can transferred to the eponymous side of a legal type-$t$
tile via one of the blue tiles.
For example, when the substitution is applied, the blue-green index on the east side of $D$ is inherited by the tile $C$,
as its west and east indices. When the substitution is applied again,
the east index of $C$ becomes the east index of $B$.
Thus, on each side the tile $D$ is similar to a  legal tile of type $t$.

However, no legal tile of type $t$ is similar to $D$,
since every legal tile of type $t$ inherits its blue-green contour from some legal blue tile,
and every legal blue tile is not similar to $D$.

Therefore, the Weak Dichotomy Lemma is false for type $t$ and tile $D$.
As shown in Section A.2,
there exists an illegal normal tile $D'$ of type $t$ in whose decomposition all tiles are legal.
Thus item (d) of the Composition Lemma is false for $\sigma D'$.
For this reason, we had to consider this case separately.
The sketch of proof presented in~\cite{fo} ignores this problem. 


\begin{thebibliography}{9}
\bibitem{gs} Chaim Goodman-Strauss, Matching rules and substitution tilings, Ann. of Math. 147 (1998) 181–223.

\bibitem{gs1} Chaim Goodman-Strauss, Aperiodic Hierarchical Tilings, in Sadoc, J. F.; Rivier, N. (eds.), Foams and Emulsions, Dordrecht: Springer, pp. 481–496 (1999)

\bibitem{gsh}  Branko Grünbaum, Geoffrey Colin Shephard, Tilings and Patterns, New York: W. H. Freeman (1987). 

\bibitem m Shahar Mozes, Tilings, substitution systems and dynamical systems generated by them, J. Analyze Math. 53 (1989) 139–186.

\bibitem{fo}Thomas Fernique, Nicolas Ollinger. Combinatorial substitutions and sofic tilings.
 arXiv:1009.5167 (2010).

\bibitem{ra}E. Arthur
Robinson Jr.,  On the table and the chair. Indagationes Mathematicae. 10:4 (1999) 581–599. 
\end{thebibliography}
\end{document}